\documentclass[11pt]{amsart}

\usepackage{amsfonts, bbm}
\usepackage{amssymb}
\usepackage{amsmath, amsthm}

\usepackage{graphicx}
\usepackage[margin=1.5in]{geometry}
\usepackage[shortlabels]{enumitem}
\usepackage{hyperref}

\newtheorem{theorem}{Theorem}[section]
\newtheorem{lemma}[theorem]{Lemma}

\newtheorem{proposition}[theorem]{Proposition}

\newtheorem{definition}[theorem]{Definition}
\theoremstyle{remark}
\newtheorem{remark}[theorem]{Remark}
\newtheorem{example}[theorem]{Example}

\newcommand{\R}{\mathbb{R}}
\newcommand{\Q}{\mathbb{Q}}
\newcommand{\Z}{\mathbb{Z}}
\newcommand{\N}{\mathbb{N}}
\newcommand{\Idem}{\mathrm{Idem}}
\newcommand{\supp}{\mathrm{supp}}
\newcommand{\Lsc}{\mathrm{Lsc}}
\newcommand{\A}{\mathrm{A}}
\newcommand{\s}{\epsilon}
\newcommand{\Hom}{\mathrm{Hom}}
\newcommand{\Cu}{\mathrm{Cu}}

\begin{document}
\title{Cones of traces arising from AF C*-algebras}

\author{Mark Moodie}
\author{Leonel Robert}


\begin{abstract}
We characterize the topological non-cancellative cones that are expressible as projective limits of finite powers of $[0,\infty]$. These are also the cones of lower semicontinuous extended-valued traces on AF C*-algebras. Our main result may be regarded as a generalization of the fact that any Choquet simplex is a projective limit of finite dimensional simplices. To obtain our main result, we first establish a duality between certain non-cancellative topological cones and Cuntz semigroups with real multiplication. This duality extends the duality between compact convex sets and complete order unit vector spaces to a non-cancellative setting. 
\end{abstract}	

\maketitle

\section{Introduction}
By a theorem of Lazar and Lindenstrauss, any Choquet simplex can be expressed as a projective limit of finite dimensional simplices (see \cite{lazar-lindenstrauss}, \cite{EHS}). This has implications for C*-algebras: given a  Choquet simplex $K$, there exists a simple, unital,  approximately finite dimensional (AF)  C*-algebra whose set of tracial states is isomorphic to $K$ (\cite{blackadar, EHS}). In the investigations on the structure of a C*-algebra, another  kind of trace is also of interest, namely, the lower semicontinuous traces with values in $[0,\infty]$. These traces form a non-cancellative topological  cone. (By cone we understand an abelian monoid endowed with a scalar multiplication by positive scalars.) Our goal here is to characterize through intrinsic properties the topological cones arising  as the lower semicontinuous  $[0,\infty]$-valued traces on an AF C*-algebra. These are also the  projective limits of cones of the form $[0,\infty]^n$, with $n\in\N$, and also, the cones arising as the $[0,\infty]$-valued monoid morphisms on the positive elements of a  dimension group.

Let $A$ be C*-algebra. Denote its cone of positive elements by $A_+$. A map $\tau\colon A_+\to [0,\infty]$ 
is called a trace if it is linear (additive, homogeneous, mapping 0 to 0) and satisfies that $\tau(x^*x)=\tau(xx^*)$
for all $x\in A$. We are interested in the lower semicontinuous traces. Let $T(A)$ denote the cone of $[0,\infty]$-valued lower semicontinuous traces on $A_+$. By the results of \cite{ERS}, $T(A)$ is a complete lattice when endowed with the algebraic order,  and addition in $T(A)$ is distributive with respect to the lattice operations. Further, one can endow $T(A)$ with a topology that is locally convex, compact and Hausdorff. We call an abstract topological cone with these properties an \emph{extended Choquet cone} (see Section \ref{sectionECCs}). 

By an AF C*-algebra we understand an inductive limit, over a possibly uncountable index set, of finite dimensional C*-algebras. Not every extended Choquet cone arises as the cone of lower semicontinuous traces on an AF C*-algebra. The requisite additional properties are sorted out in the theorem below.

An element $w$ in a cone is called idempotent if $2w=w$. Given a cone $C$, we denote by $\Idem(C)$ the set of idempotent elements of $C$.

\begin{theorem}\label{mainchar}
Let $C$ be an extended Choquet cone (see Definition \ref{defExtCho}). The following are equivalent:
\begin{enumerate}[\rm (i)]
\item
$C$ is isomorphic to $T(A)$ for some AF C*-algebra $A$. 
\item
$C$ is isomorphic to $\mathrm{Hom}(G_+,[0,\infty])$ for some dimension group $(G,G_+)$. (Here 
$\mathrm{Hom}(G_+,[0,\infty])$ denotes the set of monoid morphisms from $G_+$ to $[0,\infty]$.) 

\item
$C$ is a projective limit of cones of the form $[0,\infty]^n$, $n\in \N$.

\item
$C$ has the following properties: 
\begin{enumerate}
\item $\Idem(C)$ is an algebraic lattice under the opposite algebraic order, 
\item for each $w\in \Idem(C)$, the set $\{x\in C:x\leq w\}$ is connected.
\end{enumerate} 
\end{enumerate}
Moreover, if $C$ is metrizable and satisfies (iv), then the C*-algebra in (i) may be chosen separable, the group $G$ in (ii) may be be chosen countable, and the projective limit in (iii) may be chosen over a countable index set.
\end{theorem}

We refer to property (a) in part (iv)  as ``having an abundance of compact idempotents". The fact that the primitive spectrum of an AF C*-algebra has a basis of compact open sets makes this condition necessary.
We  call property (b) ``strong connectedness". The existence of a non-trivial trace on every simple ideal-quotient of an AF C*-algebra makes this condition necessary. In general, if a C*-algebra $A$ is such that
its primitive spectrum  has a basis of compact open sets, and every simple quotient $I/J$, where $J\subsetneq I$ are ideals of $A$, has a non-zero densely finite trace, then $T(A)$ has an abundance of compact idempotents and is strongly connected, i.e.,  properties (a) and (b) above hold. For example, if $A$ has real rank zero, stable rank one, and is exact, then these conditions are met. Theorem \ref{mainchar} then asserts the existence of an AF C*-algebra $B$ such that $T(A)\cong T(B)$.

The crucial implication in  Theorem \ref{mainchar} is (iv) implies (iii).
A reasonable approach to proving it  is to first prove that (iv) implies (ii) by directly constructing a dimension group $G$ from the cone $C$, very much in the spirit of the proof of the Lazar-Lindenstrauss theorem obtained by Effros, Handelmann, and Shen in \cite{EHS} (which, unlike the proof in \cite{lazar-lindenstrauss}, also deals with non-metrizable Choquet simpleces).  If the cone $C$ is assumed to be finitely generated, then we indeed obtain a direct construction of an ordered vector space with the Riesz property $(V,V^+)$ such that $\mathrm{Hom}(V^+,[0,\infty])$ is isomorphic to $C$. This is done in the last section of the paper. In the general case, however, such an approach has  eluded us. 

To prove Theorem \ref{mainchar} we first establish a duality between extended Choquet cones with an abundance of compact idempotents and certain abstract Cuntz semigroups. Briefly stated, this duality works as follows:
\[
C\mapsto \Lsc_\sigma(C)\hbox{ and }S\mapsto F(S).
\]
That is, to an extended Choquet cone $C$ with an abundance of compact idempotents one assigns the Cu-cone $\Lsc_\sigma(C)$ of lower semicontinuous linear functions $f\colon C\to [0,\infty]$ with ``$\sigma$-compact support". In the other direction, to a Cu-cone $S$ with an abundance of compact ideals one assigns the cone of functionals $F(S)$; see Section \ref{duality} and Theorem \ref{dualitythm}.  In the context of this duality, strong connectedness in $C$ translates into the property of weak cancellation in
$\Lsc_\sigma(C)$. We then use this arrow reversing duality  to turn the question of finding a projective limit representation for a cone into one of finding an inductive limit representation for a  Cu-cone. To achieve the latter, we follow the strategy of proof of the Effros-Handelmann-Shen theorem, adapted to the category at hand. The main technical complication here is the non-cancellative nature of Cu-cones, but this is adequately compensated by the above mentioned property of ``weak cancellation" (dual to strong connectedness).

A question that is closely related to the one addressed by Theorem \ref{mainchar} asks for a characterization of the lattices arising as (closed two-sided) ideal lattices of AF C*-algebras. For separable AF C*-algebras, this problem was solved by Bratteli and Elliott in \cite{bratteli-elliott}, and independently by Bergman in unpublished work: Any complete algebraic lattice with a countable set of compact elements is the lattice of closed two-sided ideals of a separable AF C*-algebra.  A thorough discussion of this result is given  by Goodearl and Wehrung in \cite{goodearl-wehrung}.  The cardinality restriction on the set of compact elements is necessary, as demonstrated by examples of R\r{u}\v{z}i\v{c}ka and Wehrung (\cite{ruzicka}, \cite{wehrungexample}). Now,  the lattice of closed two-sided ideals of a C*-algebra $A$ is in order reversing bijection with the lattice of idempotents of $T(A)$ via the assignment $I\mapsto \tau_I$, where $\tau_I$ is the $\{0,\infty\}$-valued trace vanishing on $I_+$. Thus, the realization of a cone $C$ in the form $T(A)$ entails the realization of $(\Idem(C),\leq^{\mathrm{op}})$ as the ideal lattice of $A$. Curiously, no  cardinality restriction is needed in Theorem \ref{mainchar} above. This demonstrates that the examples of R\r{u}\v{z}i\v{c}ka and Wehrung are also examples of algebraic lattices that cannot be realized as the lattice of idempotents of a cone $C$ satisfying any of the equivalent conditions of Theorem \ref{mainchar}.


This paper is organized as follows: In Section \ref{sectionECCs} we define extended Choquet cones and prove a number of background results on their structure. In Section \ref{sec:conesfromfunctionals} we go over three constructions---starting from a C*-algebra, a
dimension group, and a Cu-semigroup---yielding extended Choquet cones that are strongly connected and have an abundance of compact idempotents. Sections \ref{functionspaces} and \ref{duality} delve into  spaces of linear functions on extended Choquet cones with an abundance of compact idempotents. In Theorem \ref{dualitythm} we establish the above mentioned  duality  assigning to a cone $C$ the Cu-cone $\Lsc_\sigma(C)$, and conversely to a Cu-cone $S$ its cone of functionals $F(S)$. In Section \ref{proofofmainchar} we prove Theorem \ref{mainchar}. In Section \ref{fingen} we assume that the cone $C$ is finitely generated. In this case we give a direct construction of an ordered vector space with the Riesz property $(V,V^+)$ such that $C\cong \mathrm{Hom}(V^+,[0,\infty])$. The vector space $V$ is described as $\R$-valued functions on a certain spectrum of the cone $C$.

\textbf{Acknowledgement}: The second author thanks Hannes Thiel  for fruitful discussions on the topic of topological cones and for sharing his unpublished notes \cite{thiel}. 

\section{Extended Choquet Cones}\label{sectionECCs}

\subsection{Algebraically ordered compact cones}
We call cone an abelian monoid $(C,+)$ endowed  with  a scalar multiplication by positive real numbers  $(0,\infty)\times C\to C$ such that 
\begin{enumerate}[\rm (i)]
\item
the map $(t,x)\mapsto tx$ is additive on both variables,
\item
$s(tx)=(st)x$ for all $s,t\in (0,\infty)$ and $x\in C$, 
\item
$1\cdot x=x$ for all $x\in C$.
\end{enumerate}
We do not assume that the addition operation on $C$ is  cancellative. In fact, the  primary example of the cones that we investigate below is $[0,\infty]$ endowed with the obvious operations.

The algebraic pre-order on $C$ is defined as follows: $x\leq y$ if there exists $z\in C$ such that $x+z=y$. We say that $C$ is algebraically ordered if this pre-order is an order. 

We call $C$ a topological cone if it is endowed with a topology  for which the operations of addition and multiplication by positive scalars  are jointly continuous. 

\begin{definition}\label{defExtCho}
An algebraically ordered topological cone  $C$ is called an extended Choquet cone if 
\begin{enumerate}[\rm (i)]
\item
$C$ is a lattice under the algebraic order,  and the addition operation is distributive over both $\wedge$ and $\vee$: 
\begin{align*}
x + (y\wedge z) &=(x+y)\wedge (x+z),\\
x + (y\vee z)   &=(x+y)\vee (x+z),
\end{align*}
for all $x,y,z\in C$,
\item
the topology on $C$ is compact, Hausdorff, and locally convex, i.e., it has a basis of open convex sets.
\end{enumerate}
\end{definition}

\begin{remark}
It is a standard result that  in a compact algebraically ordered monoid both upward and downward directed sets converge to their supremum and infimum, respectively (\cite[Proposition 3.1]{edwards}, \cite[Proposition VI-1.3, p441]{GHK}). We shall make frequent use of this fact applied to extended Choquet cones. It readily follows from this and the existence of finite suprema and infima that extended Choquet cones are complete lattices. 
\end{remark}

\begin{remark}
By Wehrung's \cite[Theorem 3.11]{wehrung}, the algebraic and order theoretic properties of an extended Choquet cone may be summarized as saying that it is an injective object in the category of positively ordered monoids. 
\end{remark}

\begin{example}
The set $[0,\infty]$  is an extended Choquet cone when endowed with the standard operations of addition and  scalar multiplication and the standard topology. More generally, the powers $[0,\infty]$, endowed with coordinatewise operations and the product topology are extended Choquet cones.
\end{example}

Let $C$ and $D$ be extended Choquet cones. A map $\phi\colon C\to D$ is a morphism in the extended Choquet cones category if $\phi$ is linear (additive, homogeneous with respect to scalar multiplication, and mapping 0 to 0) and continuous.

\begin{theorem}
The category of extended Choquet cones has projective limits. 
\end{theorem}

\begin{proof}
Let $\{C_i:i\in I\}$, $\{\varphi_{i,j}\colon C_i\to C_j:i,j\in I \hbox{ with }j\leq i\}$, be a projective system of extended Choquet cones, where $I$ is an upward directed set. Define 
\[
C=\bigg\{(x_i)_{i}\in \prod_{i\in I}
C_i: x_{j}=\varphi_{i,j}(x_i)\hbox{ for all }i,j\in I\hbox{ with }j\leq i\bigg\}.
\] 
Endow the product $\prod_{i\in I} C_i$ with coordinatewise operations, coordinatewise order, and with the product topology; endow $C$ with the topological cone structure induced by inclusion. Let $\pi_i\colon C\to C_i$, $i\in I$, denote the projection maps. It follows from well known arguments that $\{C,\pi_i|_C:i\in I\}$ is the projective limit of the system $\{C_i,\phi_{i,j}:i,j\in I\}$ as compact Hausdorff topological cones (cf. \cite[Theorem 13]{davies}). Since  for each $i$ the topology of $C_i$ has a basis of open convex sets, the product topology on $\prod_{i} C_i$ also has a basis of open convex sets. Further, since $C$ is a convex subset of $\prod_i C_i$, the induced topology on $C$ is locally convex as well.

Let us now prove that $C$ is a lattice. The proof runs along the same lines as the one in \cite[Theorem 13]{davies} for projective limits of Choquet simplices. We show that $C$ is closed under finite suprema; the argument for finite infima is similar. Let $x=(x_i)_i$ and $y=(y_i)_i$ be in $C$. Their coordinatewise supremum exists in $\prod_{i} C_i$, but does not necessarily belong to $C$. For each $k\in I$ define $z^{(k)}\in \prod_i C_i$ by
\[
(z^{(k)})_i=\begin{cases}
\phi_{k,i}(x_k\vee y_k)& \hbox{if }i\leq k,\\
x_i\vee y_i &\hbox{otherwise}.
\end{cases}
\]
If $k'\geq k$, then 
\[
\phi_{k',k}(x_{k'}\vee y_{k'})\geq \phi_{k',k}(x_{k'})=x_k, 
\]
and similarly $\phi_{k',k}(x_{k'}\vee y_{k'})\geq y_k$, whence $\phi(x_{k'}\vee y_{k'})\geq x_k\vee y_k$. It follows that
\[
(z^{(k')})_i=\phi_{k',i}(x_{k'}\vee y_{k'})\geq \phi_{k,i}(x_k\vee y_k)=(z^{(k)})_i,
\]
for $i\leq k$, while 
\[
(z^{(k')})_i\geq x_i\vee y_i=(z^{(k)})_i
\] 
for $i\nleq k$. Thus, $(z^{(k)})_{k\in I}$ is an upward directed net. Set $x\vee y:=\lim_k z^{(k)}$, which is readily shown to belong to $C$. Then $x\vee y\geq z^{(k)}\geq x,y$ for all $k$.  Suppose that $w=(w_i)_i\in C$ is such that $w\ge x,y$. Then $w_i\ge x_i\vee y_i$ for all $i$, and further 
\[
w_i=\varphi_{k,i}(w_k)\ge\varphi_{k,i}(x_k\vee y_{k}).
\] 
Hence, $w\ge z^{(k)}$ for all $k$, and so  $w\geq x\vee y$. This proves that $x\vee y$ is in fact the supremum of $x$ and $y$ in $C$. 

Let us prove distributivity of addition over $\vee$. Let $x,y,v\in C$. Fix an index $i$.
Then
\begin{align*}
((x+v)\vee (y+v))_i &=\lim_k \phi_{k,i}((x_k+v_k)\vee (y_k+v_k))\\
&=\lim_{k} \phi_{k,i}((x_k\vee y_k)+v_k)\\
&=\lim_{k}\phi_{k,i}(x_k\vee y_k)+v_i\\
&=(x\vee y+v)_i,
\end{align*}
where we have used the distributivity of addition over $\vee$ on each coordinate and the construction of joins
in $C$ obtained above. Thus, $(x+v)\vee (y+v)=(x\vee y)+v$. Distributivity over $\wedge$ is handled similarly.
\end{proof}

\subsection{Lattice of idempotents}
Throughout this subsection  $C$ denotes an extended Choquet cone. 

An element $w\in C$ is called idempotent if $2w=w$. It  follows, using that $C$ is algebraically ordered, that $t w=w$ for all $t\in (0,\infty]$. We denote the set of idempotents of $C$ by $\Idem(C)$. The set $\Idem(C)$ is  a sub-lattice of $C$: if $w_1$ and $w_2$ are idempotents then 
\[
2(w_1\vee w_2)=(2w_1\vee 2w_2)=w_1\vee w_2,
\]
where we have used that  multiplication by $2$ is an order isomorphism. Hence, $w_1\vee w_2$ is an idempotent. Similarly,  $w_1\wedge w_2$ is shown to be an idempotent. Moreover, $w_1\vee w_2=w_1+w_2$, a fact easily established.

In the lattice $\Idem(C)$, we use the symbol $\gg$ to denote the way-below relation under the opposite order.
That is, $w_1\gg w_2$ if whenever $\inf_i v_i\leq w_2$ for a decreasing net $(v_i)_i$ in $\Idem(C)$, we have  $ v_{i_0}\leq w_1$ for some $i_0$. We call $w\in \Idem(C)$ a compact idempotent if $w\gg w$. More explicitly, $w$ is compact if whenever $\inf_i v_i\leq w$ for a decreasing net $(v_i)_i$ in $\Idem(C)$, we have  $v_{i_0}\leq w$ for some $i_0$. Note: we only use the notion of compact element in $\Idem(C)$ in the sense just defined, i.e., applied to the \emph{opposite order}.

A complete lattice is called algebraic if each of its elements is a supremum of compact elements 
(\cite[Definition I-4.2]{GHK}).   

\begin{definition} We say that an extended Choquet cone $C$ has an abundance of compact idempotents if  $(\Idem(C), \leq^{\mathrm{op}})$ is an algebraic lattice, i.e.,  every idempotent in $C$ is an infimum of compact idempotents.
\end{definition}

Let $x\in C$. Consider the set $\{z\in C:x+z=x\}$. This set is closed under addition and also closed in the topology of $C$. It follows that it has a maximum element $\s(x)$. Since $2\cdot \s(x)$ is also absorbed additively by $x$, we have $\s(x)=2\s(x)$, i.e., $\s(x)$ is an idempotent. We call $\s(x)$ the support idempotent of $x$.

\begin{lemma}\label{supportlemma}(Cf. \cite[Lemma 3.2]{edwards})
Let $x,y,z\in C$.
\begin{enumerate}[\rm (i)]
\item
$\s(x) = \lim_n\frac{1}{n}x$.
\item
If $x+z\leq y+z$ then $x+\s(z)\leq y+\s(z)$.
\end{enumerate}
\end{lemma}
\begin{proof}
(i) Observe that $w:=\lim_n\frac{1}{n}x$ exists, since the infimum of a decreasing sequence is also its limit. It is also clear that $2w=w$, and that $x+w=x$. Let $z\in C$ be such that $x+z=x$. Then $x+nz=x$, i.e., 
$\frac1n x+ z=\frac1nx$, for all $n\in \N$. Letting $n\to \infty$, we get that 
$w+z=w$, and in particular, $w\leq z$. Thus, $w$ is the largest element absorbed by $x$, i.e. $w=\s(x)$. 

(ii) This is \cite[Lemma 3.2]{edwards}. Here is the argument: We deduce, by induction, that $nx+z\leq ny+z$ for all $n\in \N$. Hence, $x+\frac1nz\leq y+\frac1nz$.  Letting $n\to \infty$ and using (i), we get $x+\s(z)\leq y+\s(z)$. 
\end{proof}

\begin{lemma}\label{supportmax}
Let $K\subseteq C$ be closed and convex. Then the map $x\mapsto \s(x)$ attains a maximum on $K$.
\end{lemma}
\begin{proof}
Let $W=\{\s(x):x\in K\}$.  Let $x_1,x_2\in K$, with $\s(x_1)=w_1$ and $\s(x_2)=w_2$.  Since $K$ is convex,  
$(x_1+x_2)/2\in K$. Since
\[
\s\Big(\frac{x_1+x_2}{2}\Big)=\lim_n \frac{1}{2n}x_1+\frac{1}{2n}x_2=\s(x_1)+\s(x_2),
\]
the set $W$ is closed under addition. For each $w\in W$, let us choose $x_w\in K$ with $\s(x_w)=w$.
By compactness of $K$, the net $(x_w)_{w\in W}$ has a convergent subnet. Say $x_{h(\lambda)}\to x\in K$, where $h\colon \Lambda\to W$ is increasing and with cofinal range. For each $\lambda$ we have $x_{h(\lambda')}+h(\lambda)=x_{h(\lambda')}$ for all $\lambda'\geq \lambda$. Passing to the limit in $\lambda'$ we get $x+h(\lambda)=x$. Since $h(\lambda)$ ranges through   a cofinal set in $W$, $x+w=x$ for all $w\in W$. Thus, $\s(\cdot)$
attains its maximum on $W$ at $x$. 
\end{proof}

\begin{lemma}\label{opensetsupport}
For each idempotent $w\in C$ the set $\{x\in C:w\gg \s(x)\}$ is open. (Recall that  $\gg$ is the way below relation 
in the lattice $(\Idem(C),\leq^{\mathrm{op}})$.)
\end{lemma}

\begin{proof}
Let $x\in C$ be such that $w\gg \s(x)$. By Lemma \ref{supportmax}, for each closed convex neighborhood $K$  of $x$, there exists $x_K\in K$ at which $\s(\cdot)$ attains its maximum. By the local convexity of $C$, the system of closed convex neighborhoods of $x$ is downward directed. It follows that $(\s(x_K))_K$ is downward directed. Moreover, $x_K\to x$, since the topology is Hausdorff. We claim that $\s(x)=\inf_K \s(x_K)$, where $K$ ranges through all the closed convex neighborhoods of $x$. Proof: Set $y=\inf_K \s(x_K)$. We have $y\leq \s(x_K)\leq \frac{x_K}{n}$ for all $K$ and $n\in \N$. Passing to the limit, first in $K$ and then in $n$, we get that $y\leq \s(x)$. On the other hand, $\s(x)\leq \s(x_K)$ for all $K$ (since $x\in K$ and $\s$ attains its maximum on $K$ at $x_K$). Thus, $\s(x)\leq y$, proving our claim.

We have $w\gg \s(x)=\inf_K \s(x_K)$. Hence, there is $K$ such that $w\gg \s(x_K)$. So, there is a neighborhood of $x$ all whose members belong to $\{z\in C:w\gg \s(z)\}$. This shows that $\{z\in C:w\gg \s(z)\}$ is open.
\end{proof}

\subsection{Cancellative subcones}
Fix an idempotent $w\in C$. Let 
\[
C_w=\{x\in C:\s(x)=w\}.
\] 
Then $C_w$ is closed under sums, scalar multiplication by positive scalars, finite infima, and finite suprema.
By Lemma \ref{supportlemma} (ii), $C_w$ is also cancellative: $x+z\leq y+z$ implies that $x\leq y$ for all $x,y,z\in C_w$. It follows that $C_w$ embeds in a vector space; namely, the abelian group of formal differences $x-y$, with $x,y\in C_w$ endowed with the unique scalar multiplication extending the scalar multiplication  on $C_w$. Let $V_w$ denote the vector space of differences $x-y$, with $x,y\in C_w$. Let $\eta\colon C_w\times C_w\to V_w$ be defined by $\eta(x,y)=x-y$. We endow $C_w$ with the topology that it receives as a subset of $C$. We endow $V_w$ with the quotient topology coming from the map $\eta$.

\begin{theorem}\label{compactbase}
Let $w\in \Idem(C)$ be a compact idempotent.  Then $V_w$ is a locally convex topological vector space whose topology  restriced to $C_w$ agrees with the topology on $C_w$. Moreover, either $C_w=\{w\}$ or $C_w$ has a compact base.
\end{theorem}
Note: A subset $B$ of a cone $T$ is called a base if for each nonzero $x\in T$ the intersection of $(0,\infty)\cdot x$ with $B$ is a singleton set.
\begin{proof}
Let us first show that the topology on $C_w$ is locally compact. Since $w$ is compact, the set $\{x\in C:w\geq \s(x)\}$ is open by Lemma \ref{opensetsupport}.   We then have that  $C_w$ is the intersection of the closed set $\{x\in C:w\leq x\}$ and the open set $\{x:w\geq \s(x)\}$. Hence, $C_w$ is locally compact in the induced topology.

We can now apply \cite[Theorem 5.3]{lawson}, which asserts that if $C_w$ is a locally compact cancellative cone, then indeed $V_w$ is a locally convex topological vector space whose topology extends that of $C_w$. Finally, by \cite[Theorem II.2.6]{alfsen}, a locally compact nontrivial cone in a locally convex topological space  has a compact base.
 \end{proof}

\subsection{Strong connectedness}
Let $v,w\in \Idem(C)$ be such that $v\leq w$. Let's say that $v$ is compact relative to $w$  if $v$ is a compact idempotent in the extended Choquet cone $\{x\in C:x\leq w\}$. Put differently, if a downward directed net $(v_i)_i$ in $C$ satisfies that $\inf_i v_i\leq v$, then $v_i\wedge w\leq v$ for some $i$.

\begin{theorem}\label{kernelpropertythm}
Let $C$ be an extended Choquet cone. The following are equivalent:
\begin{enumerate}[\rm (i)]
\item
For any two $w_1,w_2\in \Idem(C)$ such that $w_1 \leq w_2$, $w_1\neq w_2$, and $w_1$ is compact relative to $w_2$, there exists $x\in C$  such that $w_1\leq x\leq w_2$ and $x$ is not an idempotent.
\item 
The set $\{x\in C:x\leq w\}$ is connected for all $w\in \Idem(C)$.
\end{enumerate}
Moreover, if the above hold then the element $x$ in (i) may always be chosen such that $\s(x)=w_1$.
\end{theorem}

\begin{proof}
We show that the negations of (i) and (ii) are equivalent.

Not (ii)$\Rightarrow$ not (i): Suppose that $\{x\in C:x\leq w\}$ is disconnected for some idempotent $w$. Working in the cone  $\{x\in C:x\leq w\}$ as the starting extended Choquet  cone, we may assume without loss of generality that $w=\infty$ (the largest element of $C$). Let $U$ and $V$ be open disjoint sets whose union is $C$. Assume that $\infty\notin U$. Observe that totally ordered subsets of $U$ have an upper bound:
if $(x_i)_i$ is a chain then $x_i\to \sup_i x_i$, and since $U$ is closed, $\sup x_i\in U$. By Zorn's lemma, $U$ contains a maximal element $v$. Since $2v$ is connected to $v$ by the path $t\mapsto tv$ with $t\in [1,2]$, we must have that $2v=v$, i.e., $v$ is an idempotent. Let's show that $v$ is compact: Let  $(v_i)_i$ be a  decreasing net of idempotents with infimum $v$. Suppose, for the sake of contradiction, that $v_i\neq v$ for all $i$. Then $v_i\in U^c$ for all $i$. Since $U^c$ is closed and $v_i\to v$, $v\in U^c$, which is a contradiction. Thus, $v$ is compact. Let $x\in C$
be such that $v\leq x\leq \infty$. If $\s(x)=\infty$, then $x=\infty$. Suppose that  $\s(x)=v$. Since $x$ is connected to $v$ by the path $t\mapsto tx$, $t\in(0,1]$, we have $x\in U$. But $v$ is maximal in $U$. Thus, $x=v$. This proves not (i).  

Not (i)$\Rightarrow$ not (ii): Suppose  that   there exist $w_1,w_2\in \Idem(C)$ such that $w_1<w_2$, $w_1$ is relatively compact in $w_2$, and there is no non-idempotent $x\in C$ such that $w_1\leq x\leq w_2$. By Zorn's lemma, we can choose $w_2$ minimal among the idempotents such that $w_1\leq w_2$ and $w_1\neq w_2$. Then 
\begin{equation}\label{gap}
w_1\leq x\leq w_2\Rightarrow x\in \{w_1,w_2\}\hbox{ for all }x\in C.
\end{equation}
Let us show that $\{x\in C:x\leq w_2\}$ is disconnected. Let $U_{1}=\{x\in C:x\leq w_1\}$ and $U_2=\{x\in C:x\nleq w_1\}$. These sets are clearly disjoint, non-empty ($w_1\in U_1$ and $w_2\in U_2$), and cover $\{x\in C:x\leq w_2\}$. It is also clear that $U_2$ is open in $C$. Let's consider $U_1$. By \eqref{gap}, $x\in U_1$ if and only if $\s(x)\leq w_1$ and $x\leq w_2$. Further, since $w_1$ is a compact idempotent in the extended Choquet cone $\{z\in C:z\leq w_2\}$, the set $U_1$ may be described as all  $x$ in the cone  $\{z\in C:z\leq w_2\}$ such that $w_1\gg \s(x)$, where the relation $\gg$ is taken in the idempotent lattice of the cone $\{z\in C:z\leq w_2\}$. Thus, by Lemma \ref{opensetsupport} applied in the extended Choquet cone  $\{z\in C:z\leq w_2\}$, the set $U_1$ is (relatively) open in $\{x\in C:x\leq w_2\}$.

Finally, let us argue that $x$ in (i) may be chosen such that $\s(x)=w_1$: Starting from $w_1\leq w_2$, with $w_1$ relatively compact in $w_2$, choose $w_2'$ minimal element in  $\{w\in \Idem(C):w_1\leq w\leq w_2,\, w\neq w_1\}$, which exists by Zorn's lemma. Let $x\in C$ be a non-idempotent such that $w_1\leq x\leq w_2'$. Then $\s(x)\in \{w_1,w_2'\}$, but we cannot have $\s(x)=w_2'$, since this entails that $x=w_2'$. So $\s(x)=w_1$.
\end{proof}

\begin{definition} Let $C$ be an extended Choquet simplex. Let us say that $C$ is strongly connected if it satisfies either one of the equivalent properties listed in Theorem \ref{kernelpropertythm}.
\end{definition}

\begin{proposition}\label{coneslimits}
If $C$ is a projective limit of extended Choquet cones of the form $[0,\infty]^n$, then $C$ is strongly connected and has an abundance of compact idempotents.
\end{proposition}
\begin{proof}
Suppose that $C=\varprojlim \{C_i,\phi_{i,j}:i,j\in I\}$, where $C_i\cong [0,\infty]^{n_i}$ for all $i\in I$. 
A projective limit of continua (compact Hausdorff connected spaces) is again a continuum. Since each $C_i$ is a continuum, so is $C$. In particular,  $C$ is connected. If $w\in \Idem(C)$, with $w=(w_i)_i\in \prod_i C_i$, then 
\[
\{x\in C:x\leq w\}=\varprojlim \{x\in C_i:x\leq w_i\}.
\] 
Thus, the same argument shows that $\{x\in C:x\leq w\}$ is connected.

The lattice of idempotent elements of $C_i$ is finite, whence algebraic under the opposite order, for all $i$. Further, by additivity and continuity, the maps $\phi_{i,j}$ preserve directed infima and arbitrary suprema (i.e., directed suprema and arbitrary infima under the opposite order). That $\Idem(C)$ is algebraic under the opposite order can then be deduced from the fact that a projective limit of algebraic lattices is again an algebraic lattice, where the morphisms preserve  directed suprema and arbitrary infima. Let us give a direct argument instead: Let $w\in \Idem(C)$, with $w=(w_i)_i\in \prod_i C_i$. For each index $k\in I$ define $w^{(k)} \in \prod_i C_i$ as the unique element in
$C$ such that   
\[
(w^{(k)})_i=
\sup\{z\in C_i:\phi_{i,k}(z)=w_k\}\hbox{ for all }i\geq k.
\]
It is not hard to show that $(w^{(k)})_{k\in I}$ is a decreasing net in $\Idem(C)$ with infimum $w$. Moreover, from the compactness of $w_k\in \Idem(C_k)$ we deduce that $w^{(k)}\in \Idem(C)$ is compact for all $k\in I$. Thus, $\Idem(C)$
is an algebraic lattice under the opposite order.
\end{proof}
 
 \section{Cones of traces and functionals}\label{sec:conesfromfunctionals}
 Here we review various constructions giving rise to extended Choquet cones.
 
 Let $A$ be a C*-algebra. Let $A_+$ denote the cone of positive elements of $A$. A map $\tau\colon A_+\to [0,\infty]$
 is called a trace if it maps 0 to 0, it is additive, homogeneous with respect to scalar multiplication, and satisfies $\tau(x^*x)=\tau(xx^*)$ for all $x\in A$. The set of all lower semicontinuous traces on $A$ is denoted by $T(A)$. It is endowed  with the pointwise operations of addition and scalar multiplication. $T(A)$ is endowed with the topology such that a net $(\tau_i)_i$ in $T(A)$ converges to $\tau\in T(A)$ if 
 \[
\limsup \tau_i((a-\epsilon)_+) \leq \tau(a)\leq \liminf \tau_i(a)
 \]
for all $a\in A_+$ and $\epsilon>0$. By \cite[Theorems 3.3 and 3.7]{ERS}, $T(A)$ is an extended Choquet cone.
  
\begin{proposition}\label{cstarECC}
Let $A$ be a C*-algebra.
\begin{enumerate}[\rm (i)]
\item
If the primitive spectrum of $A$ has a basis of compact open sets, then $T(A)$ has an abundance of compact idempotents. In particular, this holds if $A$ has real rank zero.

\item
Suppose that  for all $J\subsetneq I\subseteq A$, closed two-sided ideals of $A$ such that $I/J$ has compact primitive spectrum, there exists a  non-zero lower semicontinuous densely finite trace on $I/J$. Then $T(A)$ is strongly connected. In particular, this holds if $A$ has stable rank one and is exact. 
\end{enumerate} 
\end{proposition}

\begin{proof}
(i) The lattice of closed two-sided ideals of $A$ is in order reversing bijection with the lattice of idempotents of $T(A)$ via the assignment $I\mapsto \tau_I$, where 
\[
\tau_I(a):=\begin{cases}
0&\hbox{ for }a\in I_+,\\
\infty&\hbox{otherwise.}
\end{cases}
\]
On the other hand, the lattice of closed two-sided ideals of $A$ is isomorphic to the lattice of open sets of the primitive spectrum of $A$ (\cite[Theorem 4.1.3]{pedersen}). Thus, the lattice of idempotents of $T(A)$ is algebraic (under the opposite order) if and only if the lattice of open sets of the primitive spectrum is algebraic. The latter is equivalent to the existence of a basis of compact open sets for the topology.

(ii) Let us check that $T(A)$ satisfies condition (i) of Theorem \ref{kernelpropertythm}. Recall that idempotents in $T(A)$ have the form $\tau_I$, where $I$ is a closed two-sided ideal. Let $I$ and $J$ be (closed, two-sided) ideals of $A$, with $J\subseteq I$, so that $\tau_I\leq \tau_J$. The property that $\tau_I$ is compact relative to $\tau_J$ means that  if $(I_i)_i$ is an upward directed net of ideals such that $J\subseteq I_i\subseteq I$ for all $i$ and $I=\bigcup I_i$, then $I=I_{i_0}$ for some $i_0$. This, in turn, is equivalent to $I/J$ having compact primitive spectrum.  By assumption, there exists  $\tau\in T(I/J)$ that is densely finite and non-zero. Pre-composed with the quotient map $\pi\colon I\to I/J$ (which maps $I_+$ onto $(I/J)_+$), $\tau$ gives rise to a trace $\tau\circ\pi\in T(I)$. Let $\tilde\tau$ be the extension of $\tau\circ\pi$ to $A_+$ such that $\tilde\tau(a)=\infty$ for all $a\in A_+\backslash I_+$. Then $\tau_I\leq \tilde \tau\leq \tau_J$ and $\tilde\tau$ is not an idempotent, as it attains values other than $\{0,\infty\}$. This proves that $T(A)$ is strongly connected.

Suppose now that $A$ has stable rank one and is exact. By the arguments from the previous paragraph, it suffices to show that if $I/J$ is a non-trivial ideal-quotient with compact primitive spectrum, then there is a nontrivial lower semicontinuous densely finite trace on $I/J$.  Observe that $I/J$ has stable rank one and is exact, since both properties pass to ideals and quotients. An exact C*-algebra of stable rank one  with compact primitive spectrum always has a nonzero, densely finite lower semicontinuous trace;  see \cite[Theorem 2.15]{rordam-cones}. 
\end{proof}


Let $(G,G_+)$ be a dimension group, i.e., an ordered abelian group that is unperforated and has the Riesz refinement property. Let $\Hom(G_+,[0,\infty])$ denote the set of all $[0,\infty]$-valued monoid morphisms on $G_+$ (i.e.,  $\lambda\colon G_+\to [0,\infty]$ additive and mapping 0 to 0). Endow $\Hom(G_+,[0,\infty])$ with pointwise cone operations and with the topology of pointwise convergence.

\begin{proposition}\label{HomG}
Let $G$ be a dimension group. Then $\Hom(G_+,[0,\infty])$ is an extended Choquet cone that is strongly connected and has an abundance of compact idempotents.
\end{proposition}

\begin{proof}
By \cite[Theorem 2.33]{wehrung}, $\Hom(G_+,[0,\infty])$ is a complete positively ordered monoid, which entails that
it is a complete lattice and that addition distributes over $\wedge$ and $\vee$. The topology on $\Hom(G_+,[0,\infty])$ is that induced by its inclusion in $[0,\infty]^{G_+}$. Since the latter  is compact and Hausdorff, so is $\Hom(G_+,[0,\infty])$. Further, since $\Hom(G_+,[0,\infty])$ is a convex subset of $[0,\infty]^{G_+}$, the induced topology is locally convex. Thus, $\Hom(G_+,[0,\infty])$ is an extended Choquet cone. To see that it is strongly connected and has an abundance of compact idempotents, we can first express $(G,G_+)$ as an inductive limit of $(\Z^n,\Z_+^n)$ using the Effros-Handelmann-Shen theorem (\cite[Theorem 2.2]{EHS}), apply the functor $\mathrm{Hom}(\cdot,[0,\infty])$ to this limit, and then apply Proposition \ref{coneslimits}. We give a direct argument in the paragraphs below.

A subgroup $I\subseteq G$ is an order ideal if $I_+ :=G_+\cap I$ is a hereditary set and $I=I_+-I_+$. Idempotent elements of $\Hom(G_+,[0,\infty])$ have the form $\lambda_I(g)=0$ if $g\in I_+$ and $\lambda_I(g)=\infty$ if $g\in G_+\backslash I_+$, for some ideal $I$.  Moreover, the map $I\mapsto \lambda_I$ is an order reversing bijection between the two lattices. It is well known that the lattice of ideals of an ordered group is algebraic. Thus, $\Hom(G_+,[0,\infty])$ has abundance of compact idempotents. 

Let us now prove strong connectedness. Let $I,J\subseteq G$ be order ideals such that $J\subsetneq I$ and $\lambda_I$ is compact relative to  $\lambda_J$. In this case, this means that $I/J$ is finitely (thus, singly) generated.  Thus, it has a finite nonzero functional $\lambda\colon (I/J)_+\to [0,\infty)$ (e.g., by \cite[Theorem 1.4]{EHS}). As in the proof of Proposition \ref{cstarECC} (ii), we define a functional on all of $G_+$ by pre-composing $\lambda$ with the quotient map $I\mapsto I/J$ and setting it equal to $\infty$ on $G_+\backslash I_+$. This produces a functional $\tilde \lambda\in \Hom(G_+,[0,\infty])$ such that $\lambda_I\leq \tilde\lambda\leq \lambda_J$ and $\tilde\lambda$ is not an idempotent.
\end{proof}

Yet another construction yielding an extended Choquet cone is the dual of a Cu-semigroup. Let us first briefly recall the definition of a Cu-semigroup. Let $S$ be a positively ordered monoid. Given $x,y\in S$, let us write $x\ll y$ (read ``$x$ is way below $y$'') if whenever $(y_n)_{n=1}^\infty$ is an increasing sequence in $S$ such that $y\leq \sup_n y_n$, there exists $n_0$ such that $x\leq y_{n_0}$. 

We call $S$ a Cu-semigroup if it satisfies the following axioms:
\begin{enumerate}
	\item[O1.] For every increasing sequence $(x_n)_n$  in $S$, the supremum $\sup_{n}x_n$ exists.
	\item[O2.] For every $x\in S$ there exists a sequence $(x_n)_n$ in $S$ such that $x_n\ll x_{n+1}$ for all $n\in\mathbb{N}$ and $x=\sup_n x_n$.
	\item[O3.] If $(x_n)_n$ and $(y_n)_n$ are increasing sequences in $S$, then $\sup_n(x_n+y_n)=\sup_nx_n+\sup_n y_n$.
	\item[O4.] If $x_i\ll y_i$ for $i=1,2$, then $x_1+x_2\ll y_1+y_2$.
\end{enumerate}

Observe that in our definition of the way-below relation above  we only consider increasing sequences $(y_n)_n$, rather than increasing nets. In the context of Cu-semigroups we always use  the symbol $\ll$ to indicate this sequential version of the way below relation.

Two additional conditions that we often impose on Cu-semigroups are the following: 
\begin{enumerate}
\item[O5.] If $x'\ll x\le y$ then there exists $z$ such that $x'+z\le y\le x+z$.
\item[O6.] If $x,y,z\in S$ are such that $x\le y+z$, then for every $x'\ll x$ there are elements $y',z'\in S$ such that $x'\le y'+z'$, $y'\le x,y$ and $z'\le x,z$.
\end{enumerate}

An ordered monoid map $\lambda\colon S\to [0,\infty]$ is called a functional on $S$ if it preserves the suprema of increasing sequences. The collection of all functionals on $S$, denoted by $F(S)$, is a  cone, with the cone operations defined pointwise. $F(S)$ is endowed with the topology such that a net $(\lambda_i)_{i\in I}$ in $F(S)$ converges to a functional $\lambda$ if 
\[
\limsup_i \lambda_i(s')\leq \lambda(s) \leq \liminf_i \lambda_i(s) 
\]
for all $s'\ll s$, in $S$. By \cite[Theorem 4.8]{ERS} and \cite[Theorem 4.1.2]{robertFS}, if $S$ is a Cu-semigroup satisfying O5 and O6, then $F(S)$ is an extended Choquet cone. In Section \ref{duality} we address the problem of what conditions on $S$ guarantee that $F(S)$ has an abundance of compact idempotents and is strongly connected.

\section{Functions on an extended Choquet cone}\label{functionspaces}
Throughout this section we let $C$ denote an extended Choquet cone with an abundance of compact idempotents, i.e., such that the lattice $(\Idem(C),\leq^{\mathrm{op}})$ is algebraic. 

\subsection{The spaces $\Lsc(C)$ and $\A(C)$}
Let us denote by $\Lsc(C)$ the set of all  functions  $f\colon C\to [0,\infty]$ that are linear (additive, homogeneous with respect  to scalar multiplication, and mapping 0 to 0) and lower semicontinuous ($f^{-1}((a,\infty])$ is open for any $a\in [0,\infty)$). The linearity of the functions in $\Lsc(C)$ implies that they are also order preserving, for if $x\leq y$ in $C$, then $y=x+z$ for some $z$, and so $f(y)=f(x)+f(z)\geq f(x)$. We endow $\Lsc(C)$ with the operations of pointwise addition and scalar multiplication, and with the pointwise order. $\Lsc(C)$ is thus an ordered cone. 
Further, the pointwise supremum of functions in  $\Lsc(C)$ is again in $\Lsc(C)$; thus, $\Lsc(C)$ is a directed complete
ordered set (dcpo).

Let us denote by $\Lsc_\sigma(C)$ the subset of $\Lsc(C)$ of functions $f\colon C\to [0,\infty]$ for which the set $f^{-1}((a,\infty])$ is $\sigma$-compact---in addition to being open---for all $a\in [0,\infty)$ (equivalently, for $a=1$, by linearity.) We denote by $\A(C)$ the functions in $\Lsc(C)$ that are continuous. Notice that $\A(C)\subseteq \Lsc_\sigma(C)$, since
\[
f^{-1}((a,\infty])=\bigcup_n f^{-1}([a+\frac1n,\infty]),
\]
and the right side is a union of closed (hence, compact) subsets of $C$.

Our goal is to show that every function in $\Lsc(C)$ ($\Lsc_\sigma(C)$) is the supremum of an increasing net (sequence) of functions in $\A(C)$. We achieve this in Theorem \ref{ACsuprema} after a number of preparatory results.

Given $f\in \Lsc(C)$, define its support  $\supp(f)\in C$ as
\[
\supp(f)=\sup\{x\in C:f(x)=0\}.
\]  
Since $f(x)=0\Rightarrow f(2x)=0$, it follows easily that $\supp(f)$ is an idempotent of $C$.

For each $w\in \Idem(C)$, let 
\[
\chi_w(x)=
\begin{cases}
0 & \hbox{ if }x\leq w,\\
\infty & \hbox{otherwise.}
\end{cases}
\]
This is a function in $\Lsc(C)$.

\begin{lemma}\label{inftyfsupp}
We have $\infty\cdot f=\chi_{\supp(f)}$, for all $f\in \Lsc(C)$. (Here
$\infty \cdot f:=\sup_{n\in \N} nf$.)
\end{lemma}

\begin{proof}
The set $\{x\in C:f(x)=0\}$ is upward directed and converges to its supremum, i.e., to $\supp(f)$. It follows, by the lower semicontinuity of $f$, that $f(\supp(f))=0$.

If $x\leq \supp(f)$, then $f(x)\leq f(\supp(f))=0$. Hence, $(\infty \cdot f)(x)=0$.
If on the other hand  $x\nleq \supp(f)$, then $f(x)\neq 0$, which implies that  $(\infty \cdot f)(x)=\infty$.
We have thus shown that $\infty\cdot f=\chi_{\supp(f)}$.
\end{proof}

Let $w\in C$ be an idempotent.  Define $\A_w(C)=\{f\in \A(C): \supp(f)=w\}$ and
\[
\A_{+}(C_w)=\{f\colon C_w\to [0,\infty):\hbox{  $f$ is continuous, linear, and $f(x)=0\Leftrightarrow x=w$}\}.
\]
(Recall that we have defined $C_w=\{x\in C:\s(x)=w\}$.)

\begin{theorem}\label{ACwbijection}
If $f\in \A(C)$ then $\supp(f)$ is a compact idempotent. Further,  given a compact idempotent $w\in \Idem(C)$, the restriction map $f\mapsto f|_{C_w}$ is an ordered cone isomorphism from $\A_w(C)$ to $\A_{+}(C_w)$.
\end{theorem}
\begin{proof}
Let $f\in \A(C)$. We have already seen that $\supp(f)$ is an idempotent. To prove its compactness, let
$(w_i)_{i\in I}$ be a downward directed family of idempotents with infimum $\supp(f)$. By the continuity of $f$, we have $\lim_i f(w_i)=f(\supp(f))=0$. But $f(w_i)\in \{0,\infty\}$ for all $i$. Therefore, there exists $i_0$ such that $f(w_i)=0$ for all $i\geq i_0$. But $\supp(f)$ is the largest element on which $f$ vanishes. Hence, $w_i=\supp(f)$ for all $i\geq i_0$. Thus, $\supp(f)$ is a compact idempotent.

Now, fix a compact idempotent $w$. Let $f\in \A_w(C)$. Clearly, $f$ is continuous and linear on $C_w$, and $f(w)=0$. Let $x\in C_w$. If $f(x)=0$, then $x\leq w$, which implies that  $x=w$. Thus, $f(x)>0$ for all $x\in C_w\backslash\{w\}$. Suppose that $f(x)=\infty$. Then $f(w)=\lim_n f(\frac 1n x)=\infty$, contradicting that $w=\supp(f)$. Thus, $f(x)<\infty$ for all $x\in C_w$. We have thus shown that $f|_{C_w}\in \A_{+}(C_w)$.

It is clear that the restriction map $\A_w(C)\ni f\mapsto f|_{C_w}\in \A_+(C_w)$ is additive and order preserving. Let us show that it is an order embedding. Let $f,g\in \A_w(C)$ be such that  $f|_{C_w}\leq g|_{C_w}$. Let $x\in C$. Suppose that  $x+w\in C_w$. Then 
\[
f(x)=f(x+w)\leq g(x+w)=g(x).
\]
If, on the other hand, $x+w\notin C_w$, then $\s(x+w)>w$. Hence, 
\[
f(x)=f(x+w)\geq f(\s(x+w))=\infty.
\]
We argue similarly that $g(x)=\infty$. Thus, $f(x)=g(x)$.

Let us finally prove surjectivity. Suppose first that $C_w=\{w\}$. Then $\A_{+}(C_w)$ consists of the zero function only. Clearly then, $\chi_w|_{C_w}=0$ and $\supp(\chi_w)=w$. It remains to show that $\chi_w$ is continuous. The set  $\chi^{-1}_w(\{\infty\})=\{x\in C:x\nleq w\}$ is  open. On the other hand,  $\chi_w^{-1}(\{0\})=\{x\in C:x\leq w\}$ agrees with $\{x\in C:\s(x)\leq w\}$ (since we have assumed that $C_w=\{w\}$). The set $\{x\in C:\s(x)\leq w\}$ is open by the compactness of $w$ (Lemma \ref{opensetsupport}). Thus, $\chi_w$ is continuous. 

Suppose now that $C_w\neq \{w\}$. Let $\tilde f\in \A_{+}(C_w)$. Define $f\colon C\to [0,\infty]$ by
\[
f(x)=\begin{cases}
\tilde{f}(x+w) &\hbox{ if }x+w\in C_{w},\\
\infty&\hbox{otherwise.}
\end{cases}
\] 
Observe that $f|_{C_w}=\tilde f$. Let us show that $f\in \A_w(C)$. To show that $\supp(f)=w$, note that 
\[
f(x)=0\Leftrightarrow \tilde f(x+w)=0\Leftrightarrow x+w=w\Leftrightarrow x\leq w.
\]
Thus, $w$ is the largest element on which $f$ vanishes, i.e., $w=\supp(f)$. We leave the not difficult verification that $f$ is linear to the reader. Let us show that $f$ is continuous. Let $(x_i)_i$
be a net in $C$ with $x_i\to x$. Suppose first that $x+w\in C_w$, i.e, $\s(x)\leq w$. Since the set
$\{y\in C:\s(y)\leq w\}$ is open (Lemma \ref{opensetsupport}), $\s(x_i)\leq w$ for large enough $i$. Therefore, 
\[
\lim_i f(x_i)=\lim_i \tilde f(x_i+w)=\tilde f(x+w)=f(x). 
\]
Now suppose that $x+w\notin C_w$, in which case $f(x)=\infty$. To show that $\lim_i f(x_i)=\infty$,
we may assume that $x_i\in C_w$ for all $i$ (otherwise $f(x_i)=\infty$ by definition). Observe also that $x_i\neq w$ for large enough $i$. Let us thus assume that $x_i\in C_w\backslash \{w\}$ for all $i$. Since $w$ is a compact idempotent, $C_{\omega}$ has a compact base $K\subseteq C_{w}\setminus\{w\}$ (Theorem \ref{compactbase}). Write $x_i=t_i\tilde{x}_i$ with $\tilde{x}_i\in K$ and $t_i>0$ for all $i$. Passing to a convergent subnet and relabelling, assume that $\tilde{x}_i\to y\in K$  and $t_i\to t\in [0,\infty]$. If $t<\infty$, then $x=\lim_it_i\tilde{x}_i=ty\in  C_{w}$, contradicting our assumption that $x+w\notin C_{w}$. Hence $t=\infty$. Let $\delta>0$ be the minimum value of  $\tilde{f}$ on the compact set $K$. Then 
\[
f(x_i)=\tilde{f}(x_i)=t_if(\tilde{x}_i)\ge t_i\delta.
\] 
Hence $f(x_i)\to\infty$, thus showing the continuity of $f$ at $x$.
\end{proof}

We will need the following theorem from \cite{edwards}:
\begin{theorem}[{\cite[Theorem 3.5]{edwards}}]\label{fromedwards}
Given $f,g\in \Lsc(C)$ there exists $f\wedge g$ and further,
\[
f\wedge \sup_i f_i=\sup_i(f\wedge f_i),
\]
for any upward directed set $(f_i)_{i\in I}$ in $\Lsc(C)$.
\end{theorem}

Recall that throughout this section $C$ denotes an extended Choquet cone with an abundance of compact idempotents.

\begin{theorem}\label{ACsuprema}
Every function in $\Lsc(C)$ is the supremum of an upward directed family of functions in $\A(C)$, and
every function in $\Lsc_\sigma(C)$ is the supremum of an increasing sequence in $\A(C)$.
\end{theorem}

\begin{proof}
Let $f\in \Lsc(C)$ and set $w=\supp(f)$. We first consider the case that $w$ is compact and then deal with the general case.

Assume that $w$ is compact. If $C_w=\{w\}$, then $f=\chi_w$. Further, $\chi_w$ is continuous, as shown in the proof of Theorem \ref{ACwbijection}. Suppose that $C_w\neq \{w\}$. Consider the restriction of $f$ to $C_w$. By \cite[Corollary I.1.4]{alfsen}, $f|_{C_w}$ is the supremum of an increasing net $(\tilde h_i)_{i}$ of linear continuous functions  $\tilde h_i\colon C_w\to \R$. Since $f|_{C_w}$ is strictly positive on $C_{w}\backslash\{w\}$, it is separated from 0 on any compact base of $C_w$. It follows that the functions $\tilde h_i$ are eventually strictly positive on $C_{w}\backslash\{w\}$. Indeed, the sets $U_{i,\delta}=\tilde h_i^{-1}((\delta,\infty])\cap C_w$, where $i\in I$ and $\delta>0$, form an upward directed open cover of $C_w\backslash \{w\}$. Thus, for some $\delta>0$ and $i_0\in I$, $\tilde h_{i}$ is greater than $\delta$ on a (fixed) compact base of $C_w$ for all $i\geq i_0$. Let us thus assume that $\tilde h_i\in \A_{+}(C_w)$ for all $i$. By Theorem \ref{ACwbijection}, each $\tilde h_i$ has a unique continuous extension to an $h_i\in \A_w(C)$. Further, $(h_i)_i$ is also an increasing net. We claim that $f=\sup_i h_i$. Let us first show that $h_i\leq f$ for all $i$. Let $x\in C$ be such that  $f(x)<\infty$. Then 
\[
0=\lim_{n}\frac 1nf(x)=f(\s(x))=0.
\] 
Hence, $\s(x)\leq w$, i.e.,  $x+w\in C_w$. We thus have that
\[
h_i(x)=\tilde h_i(x+w)\leq f(x+w)=f(x).
\]
Hence, $h_i\leq f$ for all $i$. Set $h=\sup_i h_i$. Clearly $h\leq f$. If $\s(x)\leq w$ then 
\[
h(x)=h(x+w)=\sup_i h_i(x+w)=f(x).
\] 
If, on the other hand, $\s(x)\nleq w$, then $h_i(x)=\infty$ for all $i$ and $h(x)=\infty=f(x)$. Thus, $h=f$. 

Let us now consider the case when $w$ is not compact. Define
\[
H=\{h\in \A(C):h\leq (1-\epsilon)f\hbox{ for some }\epsilon>0\}.
\]
Let us show that $H$ is upward directed and has pointwise supremum $f$. Let $h_1,h_2\in H$. Set $v_1=\supp(h_1)$ and  $v_2=\supp(h_2)$, which are compact idempotents, by Theorem \ref{ACwbijection}, and satisfy that $w\leq v_1,v_2$. Set $v=v_1\wedge v_2$, which is also compact and such that $w\leq v$. Set $g=f\wedge \chi_v$, which exists by Theorem \ref{fromedwards}. Since scalar multiplication by a non-negative scalar is an order isomorphism  on $C$, we have $tg=(tf)\wedge \chi_v$. Letting $t\to \infty$ and using Theorem \ref{fromedwards}, we get $\infty\cdot g=\chi_w\wedge \chi_v=\chi_v$. Thus, $\supp(g)=v$ (Lemma \ref{inftyfsupp}). Let $\epsilon>0$ be such that $h_1,h_2\leq (1-\epsilon)f$. Then, $h_1,h_2\leq (1-\epsilon)g$.  Since we have already established the case of compact support idempotent, there exists an increasing net $(g_i)_i$ in $\A_v(C)$ such that $g=\sup_i g_i$. By \cite[Proposition 5.1]{ERS}, $h_1,h_2\ll (1-\epsilon/2)g$ in the directed complete ordered set $\Lsc(C)$ (see also the definition of the relation $\lhd$ in the next section). Thus, there exists $i_0$ such that  $h_1,h_2\leq (1-\epsilon/2)g_{i_0}$. Now $h=(1-\epsilon/2)g_{i_0}$ belongs to $H$ and satisfies that $h_1,h_2\leq h$. This shows that $H$ is upward directed.

Let us show that $f$ is the pointwise supremum of the functions in $H$. It suffices to show that $f$ is the supremum of functions in $\A(C)$, as we can then easily arrange for the $1-\epsilon$ separation. Choose a decreasing net
of compact idempotents $(v_i)_i$ with $w=\inf v_i$ (recall that $C$ has an abundance of compact idempotents). For each fixed $i$, $f\wedge \chi_{v_i}$ has support idempotent $v_i$, which is compact. Thus, as demonstrated above, $f\wedge \chi_{v_i}$ is the supremum of an increasing  net in $\A(C)$.  But $f=\sup_i f\wedge \chi_{v_i}$ (Theorem \ref{fromedwards}). It follows that $f$ is the pointwise supremum of functions in $\A(C)$.

Finally, suppose that $f\in \Lsc_\sigma(C)$, and let us show that there is a countable set in $H$ with pointwise supremum $f$.  For each $h\in H$, let $U_{h}=h^{-1}((1,\infty])$. The sets $(U_{h})_{h\in H}$  form an open cover of  $f^{-1}((1,\infty])$. Since the latter is $\sigma$-compact, we can choose a countable set $H'\subseteq H$ such that  $(U_{h})_{h\in H'}$ is also a cover of $f^{-1}((1,\infty])$. Observe that for each $x\in C$, $f(x)>1$ if and only if $h(x)>1$ for some $h\in H'$. It follows, by the homogeneity with respect to scalar multiplication of these functions, that $\sup_{h\in H'}h(x)=f(x)$ for all $x\in C$. Now using that $H$ is upward directed we can construct an increasing sequence with supremum $f$.
\end{proof}

\begin{theorem}\label{metrizableC}
Let $C$ be a metrizable extended Choquet cone with an abundance of compact idempotents. Then 
there exists a countable subset of $\A(C)$ such that every function in $\Lsc(C)$ is the supremum
of an increasing sequence of functions in this set.
\end{theorem}	 

\begin{proof}
Let us first argue that the set of compact idempotents is countable. Let $(U_{i})_{i=1}^\infty$ be a countable basis for the topology of $C$. Let $w\in \Idem_c(C)$ be a compact
idempotent. Since $\{x\in C:w\leq x\}$ is an open set, by Lemma \ref{opensetsupport}, there exists $U_i$ such that $w\in U_i\subseteq \{x\in C:w\leq x\}$. Clearly then $w=\inf U_i$. Thus, the set of compact idempotents embeds in the countable set $\{\inf U_i:i=1,2,\ldots\}$.

Now fix a compact idempotent $w$. Recall that $\A_w(C)$ is isomorphic to the cone $\A_+(C_w)$ of positive linear functions on the cone $C_w$. Suppose that $C_w\neq \{w\}$. Let $K$ denote a compact base of $C_w$, which exists by Theorem \ref{compactbase}, and is metrizable since $C$ is metrizable by assumption. Then $\A_+(C_w)$ is separable in the metric induced by the uniform norm on $K$, since it embeds in $C(K)$, which is separable. Let $\tilde B_w\subseteq \A_+(C_w)$ be a countable dense subset.  It is not hard now to express any function in $\A_+(C_w)$ as  the supremum of an increasing sequence in $\tilde B_w$.  Indeed, it suffices to show that for any $\epsilon>0$ and $f\in \A_{+}(C_w)$, there exists $g\in \tilde B_w$ such that $(1-\epsilon)f\leq g\leq f$. Keeping in mind that $f$ is separated from 0 on $K$, we can choose $g\in\tilde B_w$ such that 
\[
\Big\|(1-\frac\epsilon{2})f|_K-g|_K\Big \|_{\infty}<\frac\epsilon{2} \min_{x\in K}|f(x)|.
\] 
Then $g$ is as desired. Let $B_w\subseteq \A_w(C)$ be the set mapping bijectively onto $\tilde B_w\subseteq \A_+(C_w)$ via the restriction map. By Theorem \ref{ACwbijection}, every function in $\A_w(C)$ is the supremum of an increasing sequence in $B_w$. If, on the other hand, $C_w=\{w\}$,  then 
$\A_w(C)=\{\chi_w\}$. In this case we set $B_w=\{\chi_w\}$.

Let $B=\bigcup_w B_w$, where $w$ ranges through the set of compact idempotents, and $B_w$ is as in the previous paragraph. Observe that $B$ is countable. Let us show that every function in $f\in \Lsc(C)$ is the supremum of an  increasing sequence in $B$. Observe that $\Lsc(C)=\Lsc_\sigma(C)$, since all open subsets of a compact metric space  are $\sigma$-compact. Thus, $f=\sup_n h_n$, where $(h_n)_{n=1}^\infty$ is an increasing sequence in $\A(C)$.  The sequence $h_n'=(1-\frac 1n)h_n$ is also increasing, with supremum $f$, and $h_{n}'\ll h_{n+1}'$
in the directed complete ordered set $\Lsc(C)$ (see \cite[Proposition 5.1]{ERS} and also the definition of the relation $\lhd$ in the next section). Say $h_{n+1}'\in \A_{w_n}(C)$ for some compact idempotent $w_n$. Since $h_{n+1}'$
is the supremum of a sequence in $B_{w_n}$, we can choose $g_n\in B_{w_n}$ such that $h_n'\leq g_n\leq h_{n+1}'$. Then $(g_n)_{n=1}^\infty$ is an increasing sequence in $B$  with supremum $f$.
\end{proof}	

\section{Duality with Cu-cones}\label{duality}
By a Cu-cone we understand a Cu-semigroup $S$ that is also a cone, i.e., it is endowed with a scalar multiplication by $(0,\infty)$ compatible with the monoid structure of $S$; see Section \ref{sectionECCs}. Further we ask that
\begin{enumerate}
\item
$t_1\leq t_2$ and $s_1\leq s_2$ imply $t_1s_1\leq t_2s_2$ for all $t_1,t_2\in (0,\infty)$ and $s_1,s_2\in S$,
\item
$\sup_n t_ns_n=(\sup_n t_n)(\sup_n s_n)$ where $(t_n)_{n=1}^\infty$ and $(s_n)_{n=1}^\infty$ are increasing sequences in $(0,\infty)$ and $S$, respectively.
\end{enumerate}	
Cu-cones are called Cu-semigroups with real multiplication in \cite{robertFS}. They are also Cu-semimodules over the Cu-semiring $[0,\infty]$, in the sense of \cite{tensorthiel}.

In this section we prove a duality between extended Choquet cones with an abundance of compact idempotents and certain Cu-cones.   Throughout this section, $S$ denotes a Cu-cone satisfying the axioms O5 and O6, so that $F(S)$ is an extended Choquet cone. 

Let us recall the relation $\lhd$ in $\Lsc(C)$ defined in \cite{ERS}: Given $f,g\in \Lsc(C)$, we write $f\lhd g$ if $f\le (1-\varepsilon)g$ for some $\varepsilon>0$ and $f$ is continuous at each $x\in C$ such that $g(x)<\infty$. 
By \cite[Proposition 5.1]{ERS}, $f\lhd g$ implies that $f$ is way below $g$ in the dcpo $\Lsc(C)$, meaning that for any upward directed net $(g_i)_i$ such that $g\leq \sup g_i$, there exists $i_0$ such that $f\leq g_{i_0}$.

\begin{lemma}({Cf. \cite[Lemma 3.3.2]{robertFS}})\label{lhdll}
Let $f,g\in \Lsc(C)$ be such that $f\lhd g$. Then here exists $h\in \Lsc(C)$ such that $f+h=g$ and $h\geq \epsilon g$
for some $\epsilon>0$. Moreover, if $f,g\in \Lsc_\sigma(C)$, then $h$ may be chosen in $\Lsc_\sigma(C)$, and if $f,g\in \A(C)$, then $h$ may be chosen  in $\A(C)$.
\end{lemma}	
\begin{proof}
Define $h\colon C\to [0,\infty]$ by	
\[
h(x) = \begin{cases}
g(x) - f(x) & \hbox{ if }g(x)<\infty,\\
\infty&\hbox{ otherwise}.
\end{cases}
\]
Then $f+h=g$. The linearity of $h$ follows from a straightforward analysis. Since $f\lhd g$, there exists $\epsilon>0$ such that $f\leq (1-\epsilon)g$. Then $g(y)-f(y)\geq \epsilon g(y)$ whenever $g(y)<\infty$, while if $g(y)=\infty$ then $g(y)=\infty=h(y)$. This establishes that $h\geq \epsilon g$.

The proof of \cite[Lemma 3.3.2]{robertFS} establishes the lower semicontinuity of $h$. Let us recall it here: Let $(x_i)_i$ be a net in $C$ such that $x_i\to x$.  Suppose first that $g(x)<\infty$. Then $f(x)<\infty$, and by the continuity of $f$ at $x$, $f(x_i)<\infty$ for large enough $i$. Then, 
\[
\liminf_i h(x_i)\geq \liminf_i g(x_i)-f(x_i)\geq g(x)-f(x)=h(x).
\]
Suppose now that $g(x)=\infty$, so that $h(x)=\infty$. Since $h\geq \epsilon g$,
\[
\liminf_i h(x_i)\geq \epsilon \liminf_i g(x_i)\geq \epsilon g(x)=\infty,
\] 
thus showing lower semicontinuity at $x$. 
	
Assume now that $f,g\in \Lsc_\sigma(C)$. It is not difficult to show that $h(x)>1$ if and only if $g(x)>1/\epsilon$ or $g(x)>1+r$ and $f(x)\leq r$ for some $r\in \Q$. Thus,
\[
h^{-1}((1,\infty])=g^{-1}((1/\epsilon,\infty])\cup \bigcup_{r\in \Q}g^{-1}((1+r,\infty])\cap f^{-1}([0,r]).
\]
The right side is $\sigma$-compact. Hence, $h\in \Lsc_\sigma(C)$.

Assume now that $f,g\in \A(C)$. Continuity at $x\in C$ such that $h(x)=\infty$ follows automatically from lower semicontinuity. Let $x\in C$ be such that $h(x)<\infty$, i.e., $g(x)<\infty$.  If $x_i\to x$ then $g(x_i)<\infty$
and $f(x_i)<\infty$ for large enough $i$. Then 
\[
h(x_i)=g(x_i)-f(x_i)\to g(x)-f(x)=h(x),\] 
where we used the continuity of $g$ and $f$. Thus, $h$ is continuous at $x$.
\end{proof}	


By an ideal of a Cu-cone we understand a  subcone that is closed under the suprema of increasing sequence.  There is an order reversing bijection between the ideals of $S$ and the idempotents of $F(S)$: 
\[
I\mapsto \lambda_I(x):=\begin{cases}
0&\hbox{if }x\in I\\
\infty&\hbox{otherwise,}
\end{cases}
\] 
where $I$ ranges through the ideals of $S$.

Let us say that a Cu-cone $S$ has an abundance of compact ideals if the lattice of ideals of $S$ is algebraic, i.e., every ideal of $S$ is a supremum of compact ideals.

\begin{theorem}\label{dualitythm}
Let $S$ be a Cu-cone satisfying O5 and O6 and having an abundance of compact ideals. Then $F(S)$ is an extended Choquet cone with an abundance of compact idempotents. Moreover, $S\cong \Lsc_\sigma(F(S))$ via the assignment 
\[
S\ni s\mapsto \hat s\in \Lsc_\sigma(F(S)),
\] 
where $\hat{s}(\lambda):=\lambda(s)$ for all $\lambda\in F(S)$. 

Let $C$ be an extended Choquet cone with an abundance of compact idempotents. Then $\Lsc_\sigma(C)$ is a Cu-cone satisfying O5 and O6 and having an abundance of compact ideals. Moreover, $C\cong F(\Lsc_\sigma(C))$ via the assignment 
\[
C\ni x\mapsto \hat x\in F(\Lsc_\sigma(C)),
\]
where $\hat x(f):=f(x)$ for all $f\in \Lsc_\sigma(C)$.
\end{theorem}

\begin{proof}
As recalled in Section \ref{sec:conesfromfunctionals}, by the results of  \cite{robertFS}, $F(S)$ is an extended Choquet cone. The bijection between the ideals of $S$ and the idempotents of $F(S)$ translates the abundance of compact ideals of $S$ directly into the abundance of compact idempotents of $F(S)$. By \cite[Theorem 3.2.1]{robertFS}, the mapping 
\[
S\ni s\mapsto \hat s\in \Lsc(F(S))
\] 
is an isomorphism of the Cu-cone $S$ onto the space of functions $f\in \Lsc(F(S))$ expressible as the pointwise supremum of an increasing  sequence $(h_n)_{n=1}^\infty$ in $\Lsc(F(S))$ such that $h_n\lhd h_{n+1}$ for all $n$. The set of all such functions is denoted by $L(F(S))$ in \cite{robertFS}. Let us show that, under our present assumptions,  $L(F(S))=\Lsc_\sigma(F(S))$. Let $f\in \Lsc(F(S))$ be such that $f=\sup h_n$, where $h_n\lhd h_{n+1}$ for all $n$. We have $\overline{h_n^{-1}((1,\infty])}\subseteq f^{-1}((1,\infty])$ for all $n$ (\cite[Proposition 5.1]{ERS}). Hence,
\[
f^{-1}((1,\infty])=\bigcup_n \overline{h_n^{-1}((1,\infty])}.
\]
Thus, $f\in\Lsc_\sigma(F(S))$. Suppose, on the other hand, that $f\in \Lsc_\sigma(F(S))$. 
Then, by Theorem \ref{ACsuprema}, there exists an increasing sequence $(h_n)_{n=1}^\infty$ in $\A(F(S))$ with supremum $f$. Clearly, $h_n'=(1-\frac1n)h_n$ is also increasing, has supremum $f$, and $h_n'\lhd h_{n+1}'$ for all $n$.  Hence, $f\in L(F(S))$. 

Let's turn now to the second part of the theorem. Let $C$ be an extended Choquet cone with an abundance of compact idempotents. Let us show that $\Lsc_\sigma(C)$ satisfies all axioms O1-O6 (Section \ref{sec:conesfromfunctionals}). Let us show first that $\Lsc_\sigma(C)$ is closed under the suprema of increasing sequences: Let $f=\sup_n f_n$, with $(f_n)_{n=1}^\infty$ an increasing sequence in $\Lsc_\sigma(C)$. Then $f^{-1}((1,\infty])=\bigcup_{n=1}^\infty f_n^{-1}((1,\infty])$.  Since the sets on the right side are $\sigma$-compact, so is the left side. Thus, $f\in \Lsc_\sigma(C)$.

Let $f\in \Lsc_\sigma(C)$, and let $(h_n)_{n=1}^\infty$ be an increasing in $\A(C)$  with  supremum $f$. Then $h_n'=(1-\frac1n)h_n$ has supremum $f$ and  $h_n'\ll h_{n+1}'$ for all $n$ (since $h_n'\lhd h_{n+1}'$). This proves O2. Axiom O3 follows at once from the fact that suprema in $\Lsc_\sigma(C)$ are taken pointwise. Suppose that $f_1\ll g_1$ and $f_2\ll g_2$. Choose  $h_1,h_2\in \A(C)$ such that $f_i\leq h_i\lhd g_i$ for $i=1,2$. Then $f_1+f_2\leq h_1+h_2\lhd g_1+g_2$, from which we deduce O4.  

Let's prove O5: Suppose that $f',f,g\in \Lsc_\sigma(C)$ are such that $f'\ll f\leq g$.  Choose $h\in \A(C)$ such that $f'\leq h\lhd f$. By Lemma \ref{lhdll}, there exists $h'\in \Lsc_\sigma(C)$ such that  $h+h'=g$. Then, $f'+h'\leq g\leq f+h'$, proving O5.

Let us prove O6. We prove the stronger property that $\Lsc_\sigma(C)$ is inf-semilattice ordered, i.e., pairwise infima  exist and addition distributes over infima. Recall that, by the results of \cite{edwards}, $\Lsc(C)$ is inf-semilattice ordered (see Theorem \ref{fromedwards}). Let us show that if $f,g\in \Lsc_\sigma(C)$, then $f\wedge g$ is also in $\Lsc_\sigma(C)$. By \cite[Lemma 3.4]{edwards}, for every $x\in C$ there exist $x_1,x_2\in C$, with   $x_1+x_2=x$, such that $(f\wedge g)(x)=f(x_1)+g(x_2)$. It is then clear that
\[
(f\wedge g)^{-1}((1,\infty])=\bigcup_{\substack{a_1,a_2\in \Q, \\ a_1+a_2>1} }
f^{-1}((a_1,\infty])\cup g^{-1}((a_2,\infty]).
\]
Since the right side is a $\sigma$-compact set, $f\wedge g\in \Lsc_\sigma(C)$.  To verify O6, suppose that  $f\leq g_1+g_2$, with $f,g_1,g_2\in \Lsc_\sigma(C)$. Then, using  the distributivity of addition over $\wedge$,  $f\leq g_1+g_2\wedge f$, which proves O6. 

Finally, let us prove that $C\ni x\mapsto \hat x\in F(\Lsc_\sigma(C))$ is an isomorphism of extended Choquet cones.
We consider injectivity first: Let $x,y\in C$ be such that $f(x)=f(y)$ for all $f\in \Lsc_\sigma(C)$. Choose $f\in \A(C)$. Passing to the limit as $n\to \infty$ in $f(\frac1n x)=f(\frac 1ny)$ we deduce that $f(\s(x))=f(\s(y))$ for all $f\in \A(C)$.
Since every function in $\Lsc(C)$ is the supremum of a directed net of functions in $\A(C)$, we have that $f(\s(x))=f(\s(y))$ for all $f\in \Lsc(C)$. Now choosing $f=\chi_w$, for  $w\in \Idem(C)$, we conclude that $\s(x)=\s(y)$, i.e., $x$ and $y$ have the same support idempotent. Set $w=\s(x)=\s(y)$. Choose a compact idempotent $v$ such that $w\leq v$. Then $x+v,y+v\in C_v$, and $f(x+y)=f(y+v)$ for all $f\in \A(C)$. By Theorem \ref{ACwbijection},
$f(x+v)=f(y+v)$ for all $f\in \A_+(C_v)$.  Recall that $C_v$ has a compact base and embeds in a locally convex Hausdorff vector space $V_v$ (Theorem \ref{compactbase}). We have $f(x+v)=f(y+v)$ for all $f\in \A_+(C_v)-\A_+(C_v)$. But $\A_+(C_v)-\A_+(C_v)$ consists of all the affine functions on $C_v$ that vanish at the origin. Thus, $f(x+v)=f(y+v)$
for all such functions, and in particular, for all continuous functionals on $V_v$. Since the weak topology on $V_v$
is Hausforff, $x+v=y+v$. Passing to the infimum over all compact idempotents $v$ such that $w\leq v$,
and using that $C$  has an abundance of compact idempotents, we conclude that $x=x+w=y+w=y$. Thus, the map $x\mapsto \hat x$ is injective.

Let us prove continuity of the map $x\mapsto \hat x$. Let $(x_i)_i$ be a net in $C$ with $x_i\to x$. Let $f',f\in \Lsc_\sigma(C)$, with $f'\ll f$. By the lower semicontinuity of $f$, we have
\[
\hat x(f)=f(x)\leq \liminf_i f(x_i)=\liminf_i \hat x_i(f).
\]
Choose $h\in \A(C)$ such that $f'\leq h\leq f$, which is possible since $f$ is supremum of an increasing sequence in $\A(C)$. Then 
\[
\limsup_i \hat x_i(f')\leq \limsup_i \hat x_i(h)=\limsup_i h(x_i)=h(x)\leq f(x)=\hat x(f).
\]
This shows that $\hat x_i\to \hat x$ in the topology of $F(\Lsc_\sigma(C))$. 

Let us prove surjectivity of the map $x\mapsto \hat x$. (Linearity is  straightforward; continuity of the inverse  is automatic from the fact that the cones are compact and Hausdorff.) The range of the map $x\mapsto \hat x$ is a compact subcone of $F(\Lsc_\sigma(C))$ that separates elements of $\Lsc_\sigma(C)$ and contains 0. By the separation theorem \cite[Corollary 4.6]{ultraCu}, it must be all of  $F(\Lsc_\sigma(C))$.
\end{proof}

Let $S$ be a Cu-cone. We say that $S$ has weak cancellation if $x+z\ll y+z$ implies $x\ll y$ for all $x,y,z\in S$.

\begin{lemma}
Let $C$ be an extended Choquet cone. Let $h,h',g\in \Lsc(C)$ be such that $h\lhd g+h'$ and  $h'\lhd h$. 
Then $\supp(g+h')$ is relatively compact in $\supp(g)$.
 \end{lemma}
\begin{proof}
Set $w_1=\supp(g+h')$ and $w_2=\supp(g)$. Let $(v_i)_i$ be a downward directed net of idempotents with $\bigwedge_i v_i\leq w_1$. Then the functions $(\chi_{v_i})_i$ form an upward directed net such that 
$g+h'\leq \chi_{w_1}\leq \sup_i \chi_{v_i}$. Since $h\lhd g+h'$, there exists $i_0$ such that $h\leq \chi_{v_{i_0}}$. We have that
\[
g+h'\leq g+h\leq \chi_{w_2}+\chi_{v_{i_0}}=\chi_{w_2\wedge v_{i_0}}.
\]
Hence, $w_2\wedge v_{i_0}\leq w_1$, which proves the lemma.	
\end{proof}

\begin{theorem}\label{weakcancellation}
Let $C$ be an extended Choquet cone with an abundance of compact idempotents. Then $C$ is strongly connected if and only if $\Lsc_\sigma(C)$ has weak cancellation. 
\end{theorem}

\begin{proof}
Suppose first that $C$ is strongly connected. Let $f,g,h\in \Lsc_\sigma(C)$ be such  that $f+h\ll g+h$. 
Choose $\lhd$-increasing sequences $(g_n)_{n=1}^\infty$ and $(h_n)_{n=1}^\infty$ in $\A(C)$ such that  $g=\sup_n g_n$ and $h=\sup_n h_n$. Then  $f+h\ll g_m+h_m$ for some $m$. We will be done once we have shown that $f\leq g_m$. 

Let $x\in C$. If $g_m(x)=\infty$, then indeed $f(x)\leq \infty=g_m(x)$. Suppose that $g_m(x)<\infty$. If $h_m(x)<\infty$, then we can cancel $h_m(x)$ in $f(x)+h_m(x)\leq g_m(x)+h_m(x)$ to obtain the desired $f(x)\leq g_m(x)$. It thus suffices to show that $g_m(x)<\infty$ implies $h_m(x)<\infty$, i.e., that $\supp(g_m)\le \supp(h_m)$. Let $w_1=\supp(g_m+h_m)$ and $w_2=\supp(g_m)$. Then $w_1\le w_2$ and $w_1$ is relatively compact in $w_2$, by the previous lemma. Suppose for the sake of contradiction that $w_1\ne w_2$. By strong connectedness, there exists $x\in C$ such that $w_1\le x\le w_2$, with $\s(x)=w_1$ and $x\ne w_1$. Then, 
\begin{align*}
h(x) &\leq g_m(x)+h_m(x)\\
	 &=h_m(x)\le (1-\delta)h(x),
\end{align*} 
for some $\delta>0$. Hence, $h(x)\in\{0,\infty\}$.  If $h(x)=0$, then $h_m(x)=g_m(x)=0$, while if $h(x)=\infty$, then
$g_m(x)+h_m(x)\geq h(x)=\infty$. In either case, we get a contradiction with $0<(g_m+h_m)(x)<\infty$, which holds by Theorem \ref{ACwbijection}.  Hence, $w_1=w_2$. We thus have that $\supp(g_m)=\supp(g_m+h_m)\le \supp(h_m)$.
	
Suppose conversely that $\Lsc_\sigma(C)$ has weak cancellation. Let $w_1\leq w_2$ be idempotents in $C$, with $w_1$ relatively compact in $w_2$, and $w_1\neq w_2$. Further, using Zorn's lemma, choose $w_2$ minimal such that $w_1\neq w_2$ and $w_1$ is relatively compact in $w_2$. Suppose for the sake of contradiction that $w_1\leq x\leq w_2$ implies $x\in \{w_1,w_2\}$. Let $D=\{x\in C:x\leq w_2\}$. Then $D$ is an extended Choquet cone and $w_1$ is a compact idempotent in $D$. Further, $D_{w_1}=\{w_1\}$. So, as shown in the course of the proof of Theorem \ref{ACwbijection}, $\chi_{w_1}|_D$ is continuous on $D$. Let $(h_i)_i\in \A(C)$ be an upward directed net with supremum $\chi_{w_1}$. Since $\chi_{w_1}|_D\lhd \chi_{w_1}|_D$, there exists $i$
such that $\chi_{w_1}|_D\leq h_i|_D$. It follows that $\chi_{w_1}\leq h_i+\chi_{w_2}$ (as functions on $C$). 
Fix an index $j\geq i$. Then 
\[
3h_j\lhd \chi_{w_1}\leq h_i+\chi_{w_2}.
\]
Now let $(l_k)_k$ be an upward directed net in $\A(C)$ with supremum $\chi_{w_2}$. Then there exists an index $k$ such that $3h_j\leq h_i+l_k$. Observe that $h_i\lhd 2h_k$. By weak cancellation in $\Lsc_\sigma(C)$, we conclude that $h_j\leq l_k$. (Note: we have used weak cancellation in the form $f+h\leq g+h'$ and $h'\ll h$ imply $f\leq g$.) Thus, $h_j\leq \chi_{w_2}$ for all $j\geq i$, implying that $\chi_{w_1}\leq \chi_{w_2}$. This contradicts that $w_1\neq w_2$.
\end{proof}

In the following section we will make use of the following form of Riesz decomposition:

\begin{theorem}\label{ACRiesz}
Let $C$ be an extended Choquet cone that is strongly connected and has an abundance of compact idempotents. Let $f,g_1,g_2\in \A(C)$ be such that  $f\lhd g_1+g_2$. Then there exist $f_1,f_2\in \A(C)$ such that $f=f_1+f_2$, $f_1\lhd g_1$, and $f_2\lhd g_2$.
\end{theorem}

\begin{proof}
Let $\epsilon>0$ be such that $f\le (1-\varepsilon)g_1+(1-\varepsilon)g_2$. Then, using the distributivity of addition over $\wedge$, 
\begin{align*}
f\le f\wedge ((1-\varepsilon)g_1)+(1-\varepsilon)g_2=(1-\varepsilon)( (f\wedge g_1)+g_2).
\end{align*}
Thus, $f\lhd (f\wedge g_1)+g_2$ (recall that $f$ is continuous). By Theorem \ref{ACsuprema}, $f\wedge g_1$ is the supremum of a net of functions in $\A(C)$. Thus, there exists $h\in \A(C)$ such that 
$f\lhd h+g_2$ and $h\lhd (f\wedge g_1)$. By Lemma \ref{lhdll}, we can find $l\in \A(C)$ such that $f=h+l$. 
Then $h+l\lhd h+g_2$. By weak cancellation in $\Lsc_\sigma(C)$ (Theorem \ref{weakcancellation}), we have that $l\lhd g_2$. Setting $f_1=h$ and $f_2=l$ yields the desired result.
\end{proof}

\section{Proof of Theorem \ref{mainchar}}\label{proofofmainchar}
Throughout this section $C$ denotes an extended Choquet cone that is strongly connected and has an abundance of compact idempotents. 

\subsection{The triangle lemma}
To prove Theorem \ref{mainchar} we follow a strategy similar to the proof of  the Effros-Handelman-Shen theorem (\cite{EHS}). The key step in this proof is establishing a ``triangle lemma'', Theorem \ref{triangletheorem} below.

\begin{lemma}\label{Cugenerators}
A linear map $\phi\colon [0,\infty]\to \Lsc_\sigma(C)$ is a Cu-morphism if and only if $\phi(\infty)=\infty \cdot \phi(1)$ and $\phi(1)\in \A(C)$.
\end{lemma}

\begin{proof}
Suppose that $\phi$ is a Cu-morphism. That $\phi(\infty)=\infty \cdot \phi(1)$ follows at once from  $\phi$ being
supremum preserving and additive. Set $f=\phi(1)$. To prove the continuity of $f$, it suffices to show that it is upper semicontinuous, since it is already lower semicontinuous by assumption. Fix $\epsilon>0$. Since $1-\epsilon \ll 1$ in $[0,\infty]$,
we have $(1-\epsilon)f \ll f$ in $\Lsc_\sigma(C)$. Choose $g\in \A(C)$  such that $(1-\epsilon)f\leq g\leq f$. Let $x_i\to x$ be a convergent net in $C$.  Then,
\[
(1-\epsilon)\limsup_i f(x_i)\leq \limsup g(x_i)=g(x)\leq f(x).
\]
Letting $\epsilon\to 0$, we get that $\limsup f(x_i)\leq f(x)$. Thus, $f$ is upper semicontinuous. 

Conversely, suppose that $\phi(1)\in \A(C)$ and $\phi(\infty)=\infty \cdot \phi(1)$. Observe that if $f\in \A(C)$ then
$\alpha f\lhd \beta f$ for all scalars $0\leq \alpha<\beta\leq \infty$. Hence, $\phi(\alpha)\ll \phi(\beta)$ in $\Lsc_\sigma(C)$ whenever $\alpha\ll \beta$ in $[0,\infty]$, i.e., $\phi$ preserves the way below relation.  The rest of the properties of $\phi$ are readily verified.
\end{proof}

The core of the proof of Theorem \ref{triangletheorem} (the ``triangle lemma") is contained in the following lemma:
\begin{lemma}\label{coretrianglelemma}
Let $\phi\colon  [0,\infty]^n\to \Lsc_\sigma(C)$ be a Cu-morphism. Let $x,y\in [0,\infty)^n\cap \Z^n$ be such that $\phi(x)\ll \phi(y)$. Then there exist $N\in \N$ and Cu-morphisms 
\[
[0,\infty]^n \stackrel{Q}{\longrightarrow} [0,\infty]^N\stackrel{\psi}{\longrightarrow} \Lsc_\sigma(C),
\]
such that $\psi Q=\phi$ and $Qx\leq Qy$. Moreover, $Q$ maps $[0,\infty)^n\cap \Z^n$ to $[0,\infty)^N\cap \Z^N$
\end{lemma}

\begin{proof}
Let $x=(x_1,\ldots,x_n)$, $y=(y_1,\ldots,y_n)$, and $\phi$ be as in the statement of the lemma. Let $(E_i)_{i=1}^n$ denote the canonical basis of $[0,\infty]^n$. Set $f_i=\phi(E_i)$ for $i=1,\ldots,n$, which belong to $\A(C)$ by Lemma \ref{Cugenerators}. Let
\[
M=\max_i |x_i-y_i|, \, n_1=\#\{i:x_i-y_i=M\}, \, n_2=\#\{i:y_i-x_i=M\}.
\] 
Let us define the degree of the triple $(\phi,x,y)$, denoted $\deg (\phi,x,y)$, as the vector $(M,n_1, n_2,n)$. We order the degrees lexicographically. We will prove the lemma by induction on the degree of the triple $(\phi, x, y)$. Let us first deal with the case  $n=1$, i.e., the domain of $\phi$ is $[0,\infty]$. Since $[0,\infty]$ is totally ordered, either $x\leq y$ or $y<x$. In the first case, setting $Q$ the identity and $\phi=\psi$ gives the result. If $y<x$, then $\phi(y)\ll \phi(x)$, which, together with $\phi(x)\ll \phi(y)$, implies  that $\phi(x)=\phi(y)$ is a compact element in $\A(C)$. The only compact element in $\A(C)$ is 0, for if $f\ll f$, then $f\ll (1-\epsilon)f$ for some $\epsilon>0$, and so $f=0$ by weak cancellation (Theorem \ref{weakcancellation}). Thus, $\phi(x)=0$, which in turn implies that $\phi=0$. We can then choose $Q$ and $\psi$ to be the 0 maps.

Suppose now that $\phi$, $x$, $y$ are as in the lemma, and that the lemma holds for all triples $(\phi', x', y')$ with smaller degree. If $x\leq y$, then we can choose $Q$ the identity map, $\phi=\psi$, and we are done. Let us thus assume that $x\nleq y$. If $x_{i_0}=y_{i_0}$ for some index $i_0$, then we can write $x=x_{i_0}E_{i_0}+\tilde x$ and $y=x_{i_0}E_{i_0}+\tilde y$, where $\tilde x,\tilde y$ belong to  $S:=\mathrm{span}(E_i)_{i\neq i_0}\cong [0,\infty]^{n-1}$. By weak cancellation, $\phi(x)\ll \phi(y)$ implies that $\phi(\tilde x)\ll \phi(\tilde y)$. Since $\tilde x,\tilde y$ belong to a space of smaller dimension,  the degree of $(\phi|_S, \tilde x,\tilde y)$ is smaller than that of $(\phi,x,y)$ ($M,n_1,n_2$ have not increased, while $n$ has decreased). By the induction hypothesis, there exist maps $\tilde Q\colon S\to [0,\infty]^N$ and $\tilde\psi\colon [0,\infty]^N\to \Lsc_\sigma(C)$ such that $\tilde Q\tilde x\leq \tilde Q\tilde y$
and $\phi|_S=\tilde\psi\tilde Q$. Define $Q\colon [0,\infty]^n\to [0,\infty]^{N+1}$ as the extension of $\tilde Q$ such that $QE_{i_0} = E_{N+1}$. Extend $\tilde \psi$ to $[0,\infty]^{N+1}$ setting $\psi(E_{N+1})=f_{i_0}$. Then $\phi=\psi Q$ and 
\[
Qx=\tilde Q\tilde x+x_{i_0}E_{N+1}\leq \tilde Q\tilde y+y_{i_0}E_{N+1}=Qy, 
\] 
thus again completing the induction step.
 
We assume in the sequel that $x_i\neq y_i$,  i.e., either $x_i<y_i$ or $x_i>y_i$, for all $i=1,\ldots,n$.
Let $I=\{i:x_i>y_i\}$ and $J=\{j:y_j>x_j\}$. Let $M_1=\max_{i\in I} x_i-y_i$ and $M_2=\max_{j\in J} y_j-x_j$. Then $M=\max (M_1,M_2)$. We break-up the rest of the proof into two cases.

\emph{Case $M_1\geq M_2$}. Using weak cancellation in 
\[
\sum_{i=1}^n x_if_i = \phi(x) \ll \phi(y)=\sum_{i=1}^n y_if_i
\]
we get
\[
\sum_{i\in I} (x_i-y_i)f_i \ll \sum_{j\in J}(y_j-x_j)f_j.
\]
Let $i_1\in I$ be such that $x_{i_1}-y_{i_1}=M_1$. From the last inequality we deduce that
\[
M_1 f_{i_1}\ll \sum_{j\in J} M_2f_j,
\]
and since $M_2\leq M_1$, we get $f_{i_1}\ll \sum_{j\in J} f_j$. By the Riesz decomposition property in $\A(C)$ (Theorem \ref{ACRiesz}), there exist $g_j,h_j\in \A(C)$, with $j\in J$, such that
\[
f_{i_1}=\sum_{j\in J} g_j\hbox{ and }f_j=g_j+h_j\hbox{ for all }j\in J.
\] 
Let $N_1=n+|J|-1$, and let us label the canonical generators of $[0,\infty]^{N_1}$ with the set $\{E_i:i=1,\ldots,n, \,i\neq i_1\}\cup\{G_j:j\in J\}$. Define $Q_1\colon [0,\infty]^{n}\to [0,\infty]^{N_1}$ as follows:
\begin{align*}
Q_1E_i &=E_i\hbox{ if }i\in I\backslash \{i_1\},\\
Q_1E_{i_1} &= \sum_{j\in J} G_j,\\
Q_1E_{j} &= E_j+G_j\hbox{ if }j\in J,
\end{align*}
and extend $Q_1$ to a Cu-cone morphism on $[0,\infty]^n$. Next, define a Cu-cone morphism $\psi_1\colon [0,\infty]^{N_1}\to \Lsc_\sigma(C)$ on the same generators as follows:
\begin{align*}
\psi_1(E_i) &= f_i,\hbox{ if }i\in I\backslash\{i_1\},\\
\psi_1(E_j) &= h_j, \hbox{ if }j\in J,\\
\psi_1(G_j) &= g_j, \hbox{ if }j\in J.
\end{align*}
It is easily checked that $\psi_1Q_1=\phi$ and that $Q_1$ maps  $[0,\infty]^n\cap \Z^n$ to  $[0,\infty]^{N_1}\cap  [0,\infty]^{N_1}$. Also,
\begin{align*}
Q_1x &=\sum_{i\in I\backslash\{i_1\}}x_iE_i + \sum_{j\in J}x_{i_1}G_j+\sum_{j\in J}x_j(E_j+G_j)\\
&=\sum_{i\neq i_1} x_iE_i + \sum_{j\in J}(x_{i_1}+x_j)G_j.
\end{align*}
Similarly,
\[
Q_1y =\sum_{i\neq i_1} y_iE_i + \sum_{j\in J}(y_{i_1}+y_j)G_j.
\]
We claim that   $\deg (\psi_1,Q_1x,Q_1y)< \deg (\phi,x,y)$. Indeed, the maximum of the differences of the coordinates ($M$ above) has not gotten larger. Moreover, the number of times that $M_1$ is attained ($n_1$ above) is smaller, since we have removed the coordinate $i_1$ and added new coordinates for which
\[
(x_{i_1}+x_j)-(y_{i_1}+y_j)=M_1 +x_j-y_j\in [0,M_1-1].
\]
By induction, the lemma holds for $(\psi_1,Q_1x,Q_1y)$. Thus, there exist Cu-morphisms 
$Q_2\colon [0,\infty]^{N_1}\to [0,\infty]^{N_2}$ and $\psi_2\colon [0,\infty]^{N_2}\to \Lsc_\sigma(C)$ such that $Q_2Q_1x\leq Q_2Q_1y$ and $\psi_1=\psi_2Q_2$. Setting $Q=Q_1Q_2$ and $\psi=\psi_2$, we get the desired result.

\emph{Case $M_2>M_1$}. This case is handled similarly to the previous case, though with a few added complications. Observe first that $M_2\geq 2$ (since $M_1\geq 1$; otherwise $x\leq y$). Choose $\epsilon>0$ such that $\phi(x)\ll (1-\epsilon)\phi(y)$. If necessary, make $\epsilon$ smaller, so that we also have 
\[
\epsilon<\min\{\frac 1{4x_i},\frac{1}{4y_j}:x_i\neq 0, y_j\neq 0\}.
\] 
Notice that this implies that 
\begin{equation}
\begin{aligned}\label{ineqsxiyi}
x_i>(1-2\epsilon)y_i&\Leftrightarrow x_i>y_i, \hbox{ for }i=1,2,\ldots,n,\\
x_i<(1-2\epsilon)y_i&\Leftrightarrow x_i<y_i, \hbox{ for }i=1,2,\ldots,n.
\end{aligned}
\end{equation}

Let $h\in \A(C)$ be such that $h+\phi(x)=(1-\epsilon)\phi(y)$, which exists by Lemma \ref{lhdll}.
Enlarge the domain of $\phi$ to $[0,\infty]^{n+1}$, labelling the new generator by $H$ ($=(0,\ldots,0,1)$),
and setting $\phi(H)=h$.  We then have $(1-2\epsilon)\phi(y)\ll h+\phi(x)$, i.e.,
\[
\sum_{i=1}^n (1-2\epsilon)y_if_i \ll h+\sum_{i=1}^n x_if_i.
\]  
Using weak cancellation and the inequalities \eqref{ineqsxiyi} we can move terms around to get
\[
\sum_{j\in J} ((1-2\epsilon)y_j-x_j)f_j \ll h+\sum_{i\in I}(x_i-(1-2\epsilon)y_i)f_i.
\]
Let $j_1\in J$ be such that $y_{j_1}-x_{j_1}=M_2$. Then
\[
((1-2\epsilon)y_{j_1}-x_{j_1})f_{j_1} \ll h +\sum_{i\in I} (x_i-(1-2\epsilon)y_i)f_i.
\]
By our choice of $\epsilon$, we have the inequalities
\[
(1-2\epsilon)y_{j_1}-x_{j_1}\geq M_2-\frac12\hbox{ and }x_{i}-(1-2\epsilon)y_i\leq M_1+\frac12\hbox{ for all }i.
\] 
Hence,
\[
(M_2-\frac12)f_{j_1}\ll h+\sum_{i\in I} (M_1+\frac12)f_i.
\]
Further, $M_1+\frac12\leq M_2-\frac{1}{2}$ (since $M_2>M_1$) and $M_2-\frac12>1$ (since $M_2\geq 2$). So
\[
f_{j_1}\ll h+\sum_{i\in I} f_i.
\]
By the Riesz decomposition property in $\A(C)$ (Theorem \ref{ACRiesz}), $f_{j_1}=h'+\sum_{i\in I}g_i$ for some  $h'\ll h$ and $g_i\ll f_i$, with $i\in I$. Let us choose $h'',h_i\in \A(C)$ such that $h=h'+h''$ and $f_i=g_i+h_i$ for all $i\in I$ (Lemma \ref{lhdll}). Label the canonical generators of the Cu-cone  $[0,\infty]^{N_1}$, where $N_1=n+|I|+1$, with the set 
\[
\{E_{j}:j=1,\dots,n,\, j\neq j_1\}\cup \{G_i:i\in I\}\cup \{H,H'\}.
\]
Define a Cu-cone morphism $Q_1\colon [0,\infty]^{n+1}\to [0,\infty]^{N_1}$  as follows:
\begin{align*}
Q_1E_j &=E_j\hbox{ for }j\in J\backslash \{j_1\},\\
Q_1E_{j_1} &= H'+\sum_{i\in I}G_i,\\
Q_1E_i &= E_i+G_i \hbox{ for }i\in I,\\
Q_1H &=H+H',
\end{align*}
Next, define a Cu-cone map $\psi_1\colon [0,\infty]^{N_1}\to \Lsc_{\sigma}(C)$ by
\begin{align*}
\psi_1 E_j &= f_j\hbox{ for }j\in J\backslash \{j_1\}\\
\psi_1 E_i &=h_i, \hbox{ for }i\in I,\\
\psi_1 G_i &= g_i,\hbox{ for }i\in I,\\ 
\psi_1 H   &=h''\hbox{ and }\psi_1 H'=h'.
\end{align*}
Now $\psi_1 Q_1E_j=f_j$ for $j\in J\backslash\{j_1\}$, and 
\[\psi_1 Q_1E_{j_1}=\psi_1\left(H'+\sum_{i\in I}G_i\right)=h'+\sum_{i\in I} g_i=f_{j_1}.\]
Also, 
\[
\psi_1 Q_1E_i=\psi_1(E_i+G_i)=h_i+g_i=f_i,\hbox{ for }i\in I.
\]
Finally, $\psi_1 Q_1H=h'+h''=h$. Thus, we have checked that $\psi_1 Q_1=\phi$. Clearly, $Q_1$ maps
integer valued vectors to integer valued vectors.

Let us examine the degree of $(\psi_1, Q_1(x+H),Q_1y)$. We have that 
\begin{align*}
Q_1(x+H)&=\sum_{j\in J\backslash\{j_1\}}x_jE_j+\sum_{i\in I}x_{j_1}G_i+x_{j_1}H'+\sum_{i\in I}x_i(E_i+G_i)+(H+H')\\
	&=\sum_{j\ne j_1}x_jE_j+\sum_{i\in I}(x_{j_1}+x_i)G_i+H+(x_{j_1}+1)H'.
\end{align*} 
Similarly, we compute that
\[
Q_1y=\sum_{j\ne j_1}y_jE_j+\sum_{i\in I}(y_{j_1}+y_i)G_i+H+y_{j_1}H'.
\]   
We claim that the $\deg(\psi_1,Q_1(x+H),Q_1y)<\deg(\phi,x,y)$. To show this we check that for the pair $(Q_1(x+H),Q_1y)$ we have that: 
\begin{enumerate}
\item
the maximum coordinates difference for the indices $i$ such that $x_i>y_i$ (number $M_1$ above) is strictly less than $M_2$, 
\item
the maximum coordinates difference  for the indices where $y_j>x_j$ is at most $M_2$, 
\item
the number of indices for which $M_2$ is attained (number $n_2$ above) has decreased relative to the pair $(x,y)$. 
\end{enumerate}
The first two points are straightforward to check. The last point follows from the fact that we have removed the coordinate $j_1$, and that for the new coordinates that we have added we have
\begin{align*}
(y_{j_1}+y_i)-(x_{j_1}+x_i)&=M_2+(y_i-x_i)\in [0,M_2-1],\\
y_{j_1}-(x_{j_1}+1)&=M_2-1<M_2.
\end{align*} 

Observe that  
\[
(\psi_1Q_1)(x+H)=h+\phi(x)=(1-\epsilon)\phi(y)\ll \phi(y)=\psi_1Q_1y.
\]
Hence, by the induction hypothesis, there exist $Q_2$ and $\psi_2$ such that $\psi_1=\psi_2Q_2$ and $Q_2Q_1(x+H)\leq Q_2Q_1y$. Then $Q=Q_2Q_1$ and $\psi=\psi_2$ are as desired, thus completing the step of the induction.
\end{proof}

\begin{theorem}\label{triangletheorem}
Let $\phi\colon [0,\infty]^n\to \Lsc_\sigma(C)$ be a Cu-morphism. Let $F\subset [0,\infty)^n$
be a finite set. Then there exist $N\in \N$ and Cu-morphisms 
\[
[0,\infty]^n \stackrel{Q}{\longrightarrow} [0,\infty]^N\stackrel{\psi}{\longrightarrow} \Lsc_\sigma(C),
\]
such that $\psi Q=\phi$, 
\[
\phi x\ll \phi y\implies Qx\ll Qy\hbox{ for all }x,y\in F,
\]
and $Q$ maps $[0,\infty]^n\cap \Z^n$ to  $[0,\infty]^N\cap \Z^{N}$.
\end{theorem}

\begin{proof}
We start by noting that given elements  $x=(x_i)_{i=1}^n$ and $y=(y_i)_{i=1}^n$ in $[0,\infty]^n$, we have $x\ll y$ if and only if  
$x_i<y_i$ or $x_i=y_i=0$ for all $i=1,\dots,n$. 

Suppose first that $F=\{x,y\}\subseteq [0,\infty)^n$ and that $\phi(x)\ll\phi(y)$. Choose $\varepsilon>0$ such that $(1+\varepsilon)\phi(x)\ll (1-\varepsilon)\phi(y)$. Choose $x',y'\in [0,\infty)^n\cap\Q^n$ such that $x\ll x'\le (1+\varepsilon)x$ and $(1-\varepsilon)y\le y'\ll y$. Then $\phi(x')\ll \phi(y')$. Let $m\in \N$ be such that $mx',my'\in [0,\infty)^n\cap\mathbb{Z}^n$.  By Lemma \ref{coretrianglelemma}, there exist  $Q,\psi$ such that $\phi=\psi Q$ and $Q(mx')\le Q(my')$, i.e., $Qx'\le Qy'$. Then 
\[
Qx\ll Qx'\le Qy'\ll Qy.
\]
Lemma \ref{coretrianglelemma} also guarantees that $Q$ maps integer valued vectors to integer valued vectors. Thus, $Q$ and $\psi$ are as desired.

To deal with an arbitrary finite set $F\subseteq [0,\infty)^n$, choose $x,y\in F$ such that $\phi(x)\ll \phi(y)$ and obtain $Q_1,\psi_1$ such that $\phi=\psi_1Q_1$ and $Q_1x\ll Q_1y$. Set $F_1=Q_1F$ and apply the same argument to a new pair $x',y'\in F_1$ to obtain maps $Q_2,\psi_2$. Continue inductively until all pairs have been exhausted. Set $Q=Q_k\cdots Q_1$ and $\psi=\psi_k$.  
\end{proof}

\subsection{Building the limit}

\begin{theorem}\label{inductiveCucones}
Let $C$ be an extended Choquet cone that is strongly connected and has an abundance of compact idempotents. Then $\Lsc_\sigma(C)$ is an inductive limit in the Cu-category of an inductive system of Cu-cones of the form $[0,\infty]^n$, $n\in \N$, and with Cu-morphisms that map integer valued vectors to integer valued vectors. Moreover, if $C$ is metrizable, then this inductive system can be chosen over a countable index set. 
\end{theorem}

\begin{proof}
For each $n=1,2,\ldots$, choose an increasing sequence  $(A_k^{(n)})_{k=1}^\infty$ of finite subsets of $[0,\infty)^n$ with  dense union in $[0,\infty]^n$.

We will construct an inductive system of Cu-cones $\{S_F,\phi_{G,F}\}$, where $F,G$ range through the finite subsets of $\A(C)$, such that $S_F\cong [0,\infty]^{n_F}$ for all $F$. We also construct Cu-morphisms $\psi_{F}\colon S_F\to \Lsc_\sigma(C)$ for all $F$, finite subset of $\A(C)$, making the overall diagram commutative. We follow closely the presentation of the proof of the Effros-Shen-Handelmann theorem given in \cite{goodearl-wehrung}, adapted to the category of Cu-cones.

For each $f\in \A(C)$, define $S_{\{f\}}=[0,\infty]$  and $\psi_{\{f\}}\colon [0,\infty]\to \Lsc_\sigma(C)$ as the Cu-morphism such that $\psi_{\{f\}}(1)=f$. Fix a finite set $F\subseteq \A(C)$. Suppose that we have defined $S_G$ and $\psi_G$ for all proper subsets $G$ of  $F$.  Set  $S^F:=\prod_G S_G$, where $G$ ranges though all proper subsets of $F$. 
Define $\phi^F\colon S^F\to \Lsc_\sigma(C)$ as 
\[
\phi^F((s_G)_G)=\sum_G \psi_G(s^G).
\]
Next, we construct $Q\colon S^F\to S_F$ and $\psi\colon S_F\to \Lsc_\sigma(C)$ using Theorem \ref{triangletheorem}. Here is how: For each $G$, proper subset of $F$, let $n_G$ be such that $S_G\cong [0,\infty]^{n_G}$. Let $A=\prod_G A^{(n_G)}_{k}$, where $k=|F|$ and where $G$ ranges through all proper subsets of $F$. Then $A$ is a finite subset of $S^F$. Let us apply Theorem \ref{triangletheorem} to $\phi^F$ and the set $A$, in order to obtain maps $Q\colon S^F\to S_F\cong [0,\infty]^{n_F}$ and $\psi\colon S_F\to \Lsc_\sigma(C)$ such that $\phi^F=\psi Q$ and 
\[
\phi^F(x)\ll \phi^F(y)\Rightarrow Qx\ll Qy\hbox{ for all }x,y\in A.
\]
 Set $\psi_{F}=\psi$, and for each proper subset $G$ of $F$, define $\phi_{G,F}\colon S_G\to S_F$ as the composition of the embedding of $S_G$ in $S^F$ with the map $Q$: 
\[
S_G\hookrightarrow S^F\stackrel{Q}{\to} S_F.
\]
Observe that $\phi_{G,F}$ maps $[0,\infty]^{n_G}\cap \Z^{n_G}$ to $[0,\infty]^{n_F}\cap \Z^{n_F}$,
as both $Q$ and  $S_G\hookrightarrow S^F$ map integer valued vectors to integer valued vectors.
Continuing in this way we obtain an  inductive system $\{S_F,\phi_{G,F}\}$, indexed by the finite subsets of $\A(C)$, and maps $\psi_F\colon S_F\to \Lsc_{\sigma}(C)$ for all $F$. By construction, the overall diagram is commutative. To show that $\Lsc_\sigma(C)$ is the inductive limit in the Cu-category of this inductive system,  we must check that
\begin{enumerate}
\item
every element in $\Lsc_\sigma(C)$ is supremum of an increasing sequence contained in the  union of the ranges of the maps $\psi_F$,
\item
for each finite set $F$ (index of the system) and elements  $x',x,y\in S_F$ such that $x'\ll x$ and $\psi_{G}(x)\leq\psi_G(y)$ in $\Lsc_\sigma(C)$, there exists $F'\supset F$ such that $\phi_{F,F'}(x')\ll \phi_{F,F'}(y)$.
\end{enumerate}
Let's check the first property. By construction, if $F=\{f\}$ then $f$ is contained in the range of $\psi_F$. Examining the construction of $\psi_F$ for arbitrary $F$, it becomes clear that $F$ is contained in the range of $\psi_F$. Thus, as $F$ ranges through all finite subsets of $\A(C)$, the union of the ranges of the maps $\psi_F$ contains $\A(C)$. Moreover, by Theorem \ref{ACsuprema}, every function in $\Lsc_\sigma(C)$ is the supremum of an increasing sequence in $\A(C)$.

Suppose that $x',x,y\in S_F$ are such that $\psi_F(x)\leq \psi_F(y)$ and  $x'\ll x$. 
Then $x'\in [0,\infty)^{n_F}$ and $\psi_F(x')\ll \psi_F(y)$. Choose $y'\ll y$ and $x'\ll x''\ll x$ such that $\psi_F(x'')\ll \psi_F(y')$. Next, choose $v,w\in A_k^{(n_F)}$ for some $k$, such that $x'\ll u \ll x''$ and $y'\ll v \ll y$. Observe then that $\psi_F(u)\ll \psi_F(v)$. Let $F'\subset \A(C)$ be a finite set such that $F\subset F'$ and $|F'|\geq k$. Then, by our construction of the inductive system, we have that $\phi_{F,F'}(u)\ll \phi_{F,F'}(v)$. This implies that $\phi_{F,F'}(x')\ll \phi_{F,F'}(y)$, thus proving the second property of an inductive limit.   

Let us address the second part of the theorem. Suppose that $C$ is metrizable. By Theorem \ref{metrizableC}, there exists a countable set $B\subseteq \A(C)$ such that every function in $\Lsc_\sigma(C)$ is the supremum of an increasing sequence
in $B$. The construction of the inductive limit for $\Lsc_\sigma(C)$ in the preceding paragraphs can be repeated mutatis mutandis, letting the index set of the inductive limit be the set of finite subsets of $B$, rather than the finite subsets of $\A(C)$. The resulting inductive limit is thus indexed by a countable set. 
\end{proof}

We are now ready to proof Theorem \ref{mainchar} from the introduction.

%

\begin{proof}[Proof of Theorem \ref{mainchar}]
(i)$\Rightarrow$(iv): An AF C*-algebra has real rank zero, stable rank one, and is exact (these properties hold for finite dimensional C*-algebras and are passed on to their inductive limits). Thus, (i) implies (iv) by Proposition \ref{cstarECC}.


(iv)$\Rightarrow$(iii): Suppose that we have (iv). By Theorem \ref{inductiveCucones}, $\Lsc_{\sigma}(C)$ is an inductive limit in the Cu-category of Cu-cones of the form $[0,\infty]^n$, with $n\in \N$. We have $F([0,\infty]^n)\cong [0,\infty]^n$
 via the map
 \[
 F([0,\infty]^n)\ni \lambda\mapsto (\lambda(E_1),\ldots,\lambda(E_n))\in [0,\infty]^n,
 \] 
 where $E_1,\ldots,E_n$ are the canonical generators of $[0,\infty]^n$.
Applying the functor $F(\cdot)$ to the inductive system with limit $\Lsc_\sigma(C)$ we obtain a projective system in the category of extended Choquet cones where each cone is isomorphic to $[0,\infty]^n$ for some $n$. By the continuity of the functor $F(\cdot)$ (\cite[Theorem 4.8]{ERS}), and the fact that $F(\Lsc_\sigma(C))\cong C$ (Theorem \ref{dualitythm}), we get (iii).

(iii)$\Rightarrow$(ii): Suppose that we have (iii). Say  $C=\varprojlim_{i\in I} ([0,\infty]^{n_i}, \alpha_{i,j})$. Observe that $\alpha_{i,j}$ maps $[0,\infty)^{n_i}$ to $[0,\infty)^{n_j}$. Indeed, the support idempotent of an element in $[0,\infty)^{n_i}$ is 0. By continuity of $\alpha_{i,j}$, the same holds for the image of these elements; thus, they belong to $[0,\infty)^{n_j}$. It follows then that $\alpha_{i,j}$ is given by multiplication by a matrix $M_{i,j}$ with non-negative finite entries: $\alpha_{i,j}(v)=M_{i,j}v$ for all $v\in [0,\infty]^{n_i}$   (in $M_{i,j}v$ we regard  $v$  as a column vector and use the rule $0\cdot \infty =0$). The transpose matrix $M_{i,j}^t$ can then be regarded as a map from $\R^{n_j}$ to $\R^{n_i}$. Let us form an inductive system of dimension groups whose objects are  $\R^{n_i}$, endowed with the coordinatewise order, with $i\in I$, and with maps $M_{i,j}^t\colon \R^{n_j}\to \R^{n_i}$. This inductive system of dimension groups  gives rise to  the original system after applying the functor $\Hom(\,\cdot\,,[0,\infty])$ to it, and making the isomorphism identifications $\mathrm{Hom}(\R_+^{n_i},[0,\infty])\cong [0,\infty]^{n_i}$. Let $G$ be its limit in the category of dimension groups ($G$ is in fact a vector space). By the continuity of the functor  $\Hom(\,\cdot\,,[0,\infty])$, we have $\Hom(G_+,[0,\infty])\cong C$.  Thus, (iii) implies (ii).

(ii)$\Rightarrow$(i): By Elliott's theorem, there exists an AF C*-algebra $A$ whose Murray-von Neumann monoid of projections $V(A)$ is isomorphic to $G_+$. The result now follows from the fact, well known to experts, that  $T(A)\cong \mathrm{Hom}(V(A),[0,\infty])$ for an AF  $A$ (where $\mathrm{Hom}(V(A),[0,\infty])$ denotes the cone of monoid morphisms). Let us sketch a proof of this fact here:
Since AF C*-algebras are exact, we have by Haagerup's theorem that 2-quasitraces on $A$, and on the ideals of $A$, are traces. We apply here the version due to Blanchard and Kirchberg that includes densely finite lower semicontinuous 2-quasitraces; see \cite[Remark 2.29 (i)]{blanchard-kirchberg}. Thus, $T(A)=QT(A)$, where $QT(A)$ denotes the cone of lower semicontinuous $[0,\infty]$-valued 2-quasitraces on $A$. Further, by \cite[Theorem 4.4]{ERS}, $QT(A)\cong F(\Cu(A))$ for any C*-algebra $A$. Thus, we must show that $F(\Cu(A))\cong \mathrm{Hom}(V(A),[0,\infty])$ when $A$ is an AF C*-algebra. Let $\Cu_c(A)$ denote the submonoid of $\Cu(A)$ of compact elements, i.e., of elements $e\in \Cu(A)$ such that $e\ll e$. By \cite[Theorem 3.5]{brown-ciuperca} of Brown and Ciuperca, for stably finite $A$ the map from $V(A)$ to $\Cu(A)$ assigning to a Murray-von Neumann class $[p]_{\mathrm{MvN}}$ the Cuntz class $[p]_{\Cu}\in \Cu(A)$ is a monoid isomorphism with $\Cu_c(A)$. This holds in particular for $A$ AF.  Thus, we must show that $F(\Cu(A))\cong \mathrm{Hom}(\Cu_c(A),[0,\infty])$. This isomorphism is given by the restriction map. Indeed, since $A$ has real rank zero and stable rank one, every element of $\Cu(A)$ is supremum of an increasing sequence of compact elements (\cite[Corollary 5]{CEI}). This shows that $\lambda\mapsto \lambda|_{\Cu_c(A)}$ is injective. To prove surjectivity, suppose that we have a monoid morphism $\tau\colon \Cu_c(A)\to [0,\infty]$. Define
\[
\lambda(x)=\sup \{\tau(e):e\leq x,\,  e\in \Cu_c(A)\}.
\]
Then $\lambda$ is readily shown to be a functional on $\Cu(A)$ that extends $\tau$. Finally, from the definition of the topology on $F(\Cu(A))$ it is evident that a convergent net $(\lambda_i)_i$ in $F(\Cu(A))$ converges pointwise on compact elements of $\Cu(A)$. This shows that the map $\lambda\mapsto \lambda|_{\Cu_c(A)}$ is continuous. Since it is a bijection between compact Hausdorff spaces, its inverse is also continuous. In summary, we have the following chain of extended Choquet cones isomorphisms when $A$ is AF:
\[
T(A)=QT(A)\cong F(\Cu(A))\cong \mathrm{Hom}(\Cu_c(A),[0,\infty])\cong \mathrm{Hom}(V(A),[0,\infty]).  
\]

Finally, suppose that $C$ is metrizable and satisfies (iv). Then, in the proof of (iv)$\Rightarrow$(iii) above,
Theorem \ref{inductiveCucones} allows us to start with an inductive limit for $\Lsc_\sigma(C)$ over a countable index set. Applying the functor $F(\cdot)$, we get a projective limit for $C$ over a countable index set.  Moreover, the Cu-morphisms in the inductive system of Theorem \ref{inductiveCucones}
map integer valued vectors to integer valued vectors. Thus, the matrices $M_{i,j}$ implementing these morphisms have nonnegative integer entries. Thus,
in the proof of (iii)$\Rightarrow$(ii) we start with $C=\varprojlim_{i\in I} ([0,\infty]^{n_i}, \alpha_{i,j})$, where $\alpha_{i,j}$ is implemented by a matrix with nonnegative integer entries. We can thus construct an inductive system $(\Z^{n_i}, M_{i,j})_{i,j\in I}$, in the category of dimension groups, whose limit is  a countable dimension group $G$ such that $\Hom(G_+,[0,\infty])\cong C$, as desired.
\end{proof}

\section{Finitely generated cones}\label{fingen}
 A cone $C$ is called finitely generated if there exists a finite set $X\subseteq C$
such that for every $x\in C$ we have $x=\sum_{i=1}^n \alpha_ix_i$ for some $\alpha_i\in (0,\infty)$ and $x_i\in X$. 
In this section we give a direct construction of an ordered vector space (over $\R$) $(V,V^+)$ with the Riesz property 
and such that $\mathrm{Hom}(V^+,[0,\infty])$ is isomorphic to a given finitely generated, strongly connected, extended Choquet cone $C$. Here $\mathrm{Hom}(V^+,[0,\infty])$ denotes the monoid morphisms from $V^+$ to $[0,\infty]$. These maps are automatically homogeneous with respect to scalar multiplication; thus, they are also cone morphisms.

\begin{lemma}\label{fingenCw}
Let $C$ be a finitely generated extended Choquet cone. Then $\Idem(C)$ is finite and for each $w\in \Idem(C)$ the sub-cone $C_w$ is either isomorphic to $\{0\}$ or to $[0,\infty)^{n_w}$ for some $n_w\in \N$. (Recall that we have defined $C_w=\{x\in C:\s(x)=w\}$.) 
\end{lemma}
\begin{proof}
Let $Z$ be a finite set that generates $C$. Let $w\in C$ be an idempotent,  and write $w=\sum_{i=1}^n \alpha_ix_i$, with $x_i\in Z$ and $\alpha_i\in (0,\infty)$. Multiplying both sides  by a scalar $\delta>0$ and passing to the limit as $\delta\to 0$, we get that $w$ is the sum of support idempotents of elements in $Z$. It follows that $\Idem(C)$ is finite.

Next, let $w\in \Idem(C)$. Define $Z_w=\{x+w:x\in Z\hbox{ and }\s(x)\leq w\}$, which is a finite subset of $C_w$. We claim that $Z_w$ generates $C_w$ as a cone. Indeed, let $x\in C_w$ and write $x=\sum_{i=1}^n \alpha_i x_i$, with $x_i\in Z$ and $\alpha_i\in (0,\infty)$. Adding $w$ on both sides we get $x=\sum_{i=1}^n \alpha_i(x_i+w)$. Since  $\s(x_i)\leq \s(x)=w$, the elements $x_i+w$ are in $Z_w$. If $Z_w=\{w\}$ then $C_w$ is isomorphic to $\{0\}$. Suppose that $Z_w\neq \{w\}$. Since $w$ is a compact idempotent, $C_w$ has a compact base $K$ which is a Choquet simplex (Theorem \ref{compactbase}). Further, $K$ is finitely generated (by the set  $(0,\infty)\cdot Z_w\cap K$). Hence, $K$ has finitely many extreme points, which in turn implies that $C_w\cong [0,\infty)^{n_w}$ for some $n_w\in \N$.  
\end{proof}

\emph{For the remainder of this section we assume that $C$ is a finitely generated, strongly connected,  extended Choquet cone.} Thus, each idempotent $w\in \Idem(C)$ is compact and, by strong connectedness,  $C_w\neq \{w\}$ for all $w\neq \infty$ (here $\infty$ denotes the largest element in $C$). 

Let $w\in \Idem(C)$ and $x\in C_w$. If $z\in C$ is such that $z+w=x$, we call $z$ and extension of $x$. The set of extensions of $x$ is downward directed: if $z_1$ and $z_2$ are extensions of $x$, then so is $z_1\wedge z_2$. Consider the element $\tilde x=\inf\{z\in C:z+w=x\}$. By the continuity of addition, $\tilde x$ is also an extension of $x$, which we call the minimum extension. 

\begin{lemma}\label{irreducibles}
Let $w \in \Idem(C)$. Let $x\in C_w\backslash\{w\}$ be an element generating an extreme ray in $C_w$, and let $\tilde x$ denote the minimum extension of $x$. 

\begin{enumerate}[\rm (i)]
\item
$\tilde x$ generates an extreme ray in $C_{\s(\tilde x)}$.

\item
If $y,z\in C$ are such that $y+z=\tilde x$, then either $y\leq z$ or $z\leq y$.
\end{enumerate}
\end{lemma}

\begin{proof}
Set $v=\s(\tilde x)$.

(i) Let $y,z\in C_v$ be such that $y+z=\tilde x$. Adding $w$ on both sides we get  $(y+w)+(z+w)=x$. Since 
$y+w,z+w\in C_w$, and $x$ generates an extreme ray in $C_w$, both $y+w$ and $z+w$ are either positive scalar multiples of $x$ or equal to $w$. Assume that $y+w=w$ and $z+w=x$. The latter says that $z$ is an extension of $x$. Hence $y+z=\tilde x\leq z$ in $C_v$. By cancellation in $C_v$ (Lemma \ref{supportlemma}), we get $y=v$ and $z=\tilde x$. Suppose on the other hand that $y+w=\alpha x$ and $z+w=\beta x$ for positive scalars $\alpha,\beta$ such that $\alpha+\beta=1$. Then $y/\alpha$ and $z/\beta$ are extensions of $x$. We deduce that  $\alpha \tilde x\leq y$ and $\beta \tilde x\leq z$. Hence,
\[
\alpha \tilde x + z \leq y+z=\tilde x=\alpha \tilde x+\beta \tilde x.
\]  
By cancellation in $C_v$, $z\leq \beta\tilde x$, and so $z=\beta\tilde x$. Similarly, $y=\alpha\tilde x$. Thus, $\tilde x$ generates an extreme ray in $C_v$.

(ii) The argument is similar to the one used in (ii). After arriving at $y+w=\alpha x$ and $z+w=\beta x$, we assume without loss of generality that $\alpha\leq \frac 12\leq \beta$. Using again that $\tilde x$ is the minimum extension of $x$, we get $z\geq \tilde x/2\geq y/2+z/2$, and applying Lemma \ref{supportlemma} (ii), we arrive at $z/2\geq y/2$.
\end{proof}

\begin{remark}
The property of $\tilde x$ in Lemma \ref{irreducibles} (ii) says that $\tilde x$ is an irreducible element of the cone $C$ in the
sense defined by Thiel in \cite{thiel}.
\end{remark}

Next, we construct a suitable set of generators of $C$.  For each $w\in \Idem(C)$, let $X_w$ denote the set of minimal extensions of  all elements $x\in C_w\backslash\{w\}$  that generate an extreme ray in $C_w$. Consider the set $\bigcup_{w\in \Idem(C)} X_w$, which is closed under scalar multiplication. We form a set $X$ by picking a representative from each ray in $\bigcup_{w\in \Idem(C)} X_w$.  

\begin{proposition}\label{Xrepresentation}
Let $X\subseteq C$ be as described in the paragraph above. Each $y\in C$ has a unique representation of the form
\[
y=\sum_{i=1}^n \alpha_ix_i + w,
\]
where $x_i\in X$ and $\alpha_i\in (0,\infty)$ for all $i$, and $w\in \Idem(C)$ is such that
$\s(x_i)\leq w$ but $x_i\nleq w$ for all $i$.
\end{proposition}

\begin{proof}
Let $y\in C$, and set $w=\s(y)$. If $y=w$ then its representation is simply $y=w$. Suppose that $y\neq w$. In $C_w$, express $y$ as a sum of elements that lie in extreme rays (Lemma \ref{fingenCw}). By the construction of $X$, these elements have the form $\alpha_i(x_i+w)$, with $x_i\in X$ and $\alpha_i\in (0,\infty)$.  We thus have that 
\[
y=\sum_{i=1}^n\alpha_i(x_i+w) =\sum_{i=1}^n \alpha_i x_i + w.
\] 
We have $x_i+w\in C_w\backslash\{w\}$ for all $i$; equivalently, $\s(x_i)\leq w$ and $x_i\nleq w$ for all $i$. Thus, this is the desired representation.

To prove uniqueness of the representation, suppose that 
\[
y=\sum_{i\in I} \alpha_i x_i + w=\sum_{j\in J} \beta_j x_j + w'.  
\]
Since $\s(x_i)\leq w$ for all $i$, the support of $y$ is $w$. Thus, $w=w'$. We can now rewrite the equation above as  
\[
y=\sum_{i\in I} \alpha_i (x_i + w)=\sum_{j\in J} \beta_j (x_j + w).  
\]
This equation occurs in $C_w\cong [0,\infty)^{n_w}$. Further, $x_i+w$ and $x_j+w$ generate extreme rays of $C_w$ for all $i,j$. It follows that $I=J$ and that the two representations are the same up to relabeling of the terms. 
\end{proof}

\subsection{Constructing the vector space}

We continue to denote by $X$ the subset of $C$ defined in the previous subsection. For each $w\in \Idem(C)$, define 
\[
O_w=\{x\in X: x\nleq w\}.
\]

\begin{lemma}\label{Otopology}
Let $w_1,w_2\in \Idem(C)$. Then 
\begin{enumerate}[\rm (i)]
\item
$O_{w_1}\cup O_{w_2}=O_{w_1\wedge w_2}$.

\item
$O_{w_1}\cap O_{w_2}=O_{w_1+w_2}$.

\item
$O_{w_1}\subseteq O_{w_2}$ if and only if $w_1\geq w_2$.
\end{enumerate}
\end{lemma}

\begin{proof}
(i) It is more straightforward to work with the complements of the sets: $x\notin O_{w_1\wedge w_2}$ if and only if  $x\leq w_1\wedge w_2$, if and only if $x\leq w_1$ and $x\leq w_2$, i.e., $x\notin O_{w_1}$ and $x\notin O_{w_2}$.

(ii) Again, we work with complements. Let's show that $O_{w_1+w_2}^c\subseteq O_{w_1}^c\cup O_{w_2}^c$ (the opposite inclusion is clear). Let $x\in O_{w_1+w_2}^c$, i.e., $x\leq w_1+w_2$. Choose $z$ such that $x\wedge w_1+z=x$. Recall that the elements of $X$ are minimal extensions of non-idempotent elements that generate an extreme ray. Thus,  
by Lemma \ref{irreducibles} (ii), either $x\wedge w_1\leq z$ or $z\leq x\wedge w_1$. If $z\leq x\wedge w_1$,
then 
\[
x=x\wedge w_1+z\leq 2(x\wedge w_1)\leq w_1.
\] 
Hence $x\in O_{w_1}^c$, and we are done. Suppose that $x\wedge w_1\leq z$. It follows that $2(x\wedge w_1)\leq x$. Now repeat the same argument with $x$ and $w_2$. We are done unless we also have that $2(x\wedge w_2)\leq x$. In this case, adding the inequalities we get 
$2(x\wedge w_1)+ 2(x\wedge w_2)\leq 2x$, i.e., $x\wedge w_1 + x\wedge w_2 \leq x$. 
But $x\leq x\wedge w_1+x\wedge w_2$ (since $x\leq w_1+w_2$). Hence,  $x=x\wedge w_1+x\wedge w_2$. Applying Lemma \ref{irreducibles} (ii) again we get that either $x\leq 2(x\wedge w_1)\leq w_1$ or   $x\leq 2(x\wedge w_2)\leq w_2$. Hence, $x\in O_{w_1}^c\cup O_{w_2}^c$, as desired.

(iii) Suppose that $O_{w_1}\subseteq O_{w_2}$. By (i), $O_{w_1\wedge w_2}=O_{w_1}\cup O_{w_2}=O_{w_2}$. Assume, for the sake of contradiction, that $w_1\wedge w_2\neq w_2$. Since $C$ is strongly connected, there exists $x\in C_{w_1\wedge w_2}\setminus\{w_1\wedge w_2\}$ such that $x \leq w_2$. We can choose $x$ in an extreme ray of $C_{w_1\wedge w_2}$, since the set of all $x\in C_{w_1\wedge w_2}$ such that $x\leq w_2$ is a face. Consider the minimum extension $\tilde x$ of $x$. Adjusting $x$ by a scalar multiple, we may assume that $\tilde x\in X$. Now $\tilde x\leq w_2$, i.e, $\tilde x\notin O_{w_2}$. But we cannot have $\tilde x\leq w_1\wedge w_2$, since this would imply that 
\[
x=\tilde x+w_1\wedge w_2 = w_1\wedge w_2.
\] 
Thus, $x\in O_{w_1\wedge w_2}$. This contradicts that  $O_{w_1\wedge w_2}=O_{w_2}$.   
\end{proof}


Let $w\in \Idem(C)$. Define
\begin{align*}
P_w &=\{x\in O_w:\s(x)\leq w\},\\
\widetilde P_{w} &=P_w\cup O_w^c=\{x\in X:\s(x)\leq w\}.
\end{align*}
Observe that if $y\in C$, and $y=\sum_{i=1}^{n}\alpha_ix_i+w$ is the representation of $y$ described in Proposition \ref{Xrepresentation}, then $x_i\in P_{w}$ for all $i, 1\le i\le n$.

\begin{lemma}\label{MTconditions}
Let $w_1,w_2\in \Idem(C)$. The following statements hold:
\begin{enumerate}[\rm (i)]
\item
$\widetilde P_{w_1\wedge w_2}=\widetilde P_{w_1}\cap \tilde P_{w_2}$.
\item
If $w_1\ngeq w_2$ then $P_{w_1}\backslash O_{w_2}\neq \varnothing$.
\end{enumerate}
\end{lemma}

\begin{proof}
(i) This is straightforward: $\s(x)\le w_1$ and $\s(x)\leq w_2$ if and only if $\s(x)\le w_1\wedge w_2$.
	
(ii) Suppose that $w_1\not\ge w_2$. Let $w_3=w_1+w_2$. By Lemma \ref{Otopology} (ii), $O_{w_1}\cap O_{w_2}=O_{w_3}$. Also $w_1\le w_3$ and $w_1\ne w_3$. Since $C$ is strongly connected, there exists $y\in C_{w_1}\setminus\{w_1\}$ 
such that $w_1\le y\le w_3$. Choose $y$ on  an extreme ray (always possible, since the set of all $y\in C_{w_1}$ such that $y\le w_3$ is a face) and adjust it by a scalar so that its minimum extension $\tilde y$ belongs to $X$. Since $\tilde y+w_1\in C_{w_1}\backslash\{w_1\}$, we have that $\tilde y\nleq w_1$ and $\s(\tilde y)\leq w_1$.
That is, $\tilde y\in P_{w_1}$. Since $\tilde y\le w_3$, we also have that $\tilde y\in O_{w_3}^c\subseteq O_{w_2}^c$. We have thus obtained an element $\tilde y\in P_{w_1}\setminus O_{w_2}$.
\end{proof}

Let us say that a function $f\colon X\to \R$ is positive provided that there exists $w\in \Idem(C)$ such that $f(x)=0$ for $x\notin O_w$ and $f(x)>0$ for $x\in P_w$. We call $w$ the support of $f$ and denote it by $\supp(f)$.

\begin{lemma}\label{supportV}
The support of a positive function is unique. Further, if $f,g\colon X\to \R$ are positive then 
$\supp(f+g)=\supp(f)\wedge \supp(g)$.
\end{lemma}

\begin{proof}
Let $w_1,w_2\in \Idem(C)$ be both supports of $f$. Suppose that $w_1\ne w_2$, and without loss of generality, that $w_1\not\ge w_2$. Then there exists $x\in P_{w_1}\cap O_{w_2}^c$ (by Lemma \ref{MTconditions}). On one hand, $x\in P_{w_1}$ implies that $f(x)>0$. On the other hand, $x\in O_{w_2}^c$ implies that $f(x)=0$, a contradiction. Thus $w_1=w_2$, whereby proving the first part of the lemma.
	
To prove the second part, assume that $f$ and $g$ are positive functions on $X$, and set  $v= \supp(f)$ and $w= \supp(g)$. Clearly $f+g$ vanishes on $ O_{v}^c\cap O_{w}^c= O_{v\wedge w}^c$. Let $x\in P_{v\wedge w}$. Then, by Lemma \ref{MTconditions} (i), $x\in \tilde{P}_{v}\cap\tilde{P}_{w}$. Thus, $x$ is in one of the following sets: $P_{v}\cap P_{w}$, $P_{\nu}\cap O_{w}^c$, or $P_{w}\cap O_{\nu}^c$. In all cases we see that $(f+g)(x)>0$. Indeed, if $x\in P_{\nu}\cap P_{w}$ then $f(x),g(x)>0$; if $x\in  P_{\nu}\cap O_{w}^c$ then $f(x)>0$ and $g(x)=0$; if $x\in P_{w}\cap O_{\nu}^c$ then $f(x)=0$ and $g(x)>0$. Therefore $ \supp(f+g)=v\wedge w$.
\end{proof}

Let us denote by $V_C$ the vector space  of $\R$-valued functions on $X$ and by $V_C^+$ the set of positive functions in $V_C$. 

\begin{theorem}
The pair $(V_C,V_C^+)$ is an ordered vector space having the Riesz interpolation property.
\end{theorem}

\begin{proof}
By the previous lemma, $V_C^+$ is closed under addition. Clearly, $V_C^+$ is closed under multiplication by positive scalars. Since the pointwise strictly positive functions belong to $V_C^+$ and span $V_C$, we have   $V_C^+-V_C^+=V_C$. Also, $V_C^+\cap -V_C^+=\{0\}$, for if $f$ and $-f$ are positive then, by the previous lemma, 
\[
\supp(f)\geq \supp(f+-f)=\supp(0)=\infty,
\] 
which implies that $f=0$. Thus,  $(V_C,V_C^+)$ is an ordered vector space. 

In  \cite{maloney-tikuisis}, Maloney and Tikuisis obtained conditions guaranteeing that the Riesz interpolation property holds in a finite dimensional ordered vector space. The properties of the sets $P_w$ obtained in Lemma \ref{MTconditions} (i) and (ii) are precisely those properties in \cite[Corollary 5.1]{maloney-tikuisis} shown to guarantee that the Riesz interpolation property  holds in $(V_C,V_C^+)$.
\end{proof}

Let us define a pairing $(\cdot,\cdot)\colon C\times V_C^+\to [0,\infty]$ as follows: for each $y\in C$ and $f\in V_C^+$, write  $y=\sum_{i=1}^n \alpha_i x_i + w$, the representation of $y$ described in Proposition \ref{Xrepresentation}, and then set
\[
(y,f)=
\begin{cases}
\sum\limits_{i=1}^{n}\alpha_if(x_i) & \hbox{if }w\le \supp(f),\\
\infty &\hbox{otherwise}.
\end{cases}
\]

\begin{theorem}
The pairing defined above is bilinear. Moreover, the map $x\mapsto (x,\cdot)$, from $C$
to $\mathrm{Hom}(V_C^+,[0,\infty])$, is an isomorphism of extended Choquet cones.
\end{theorem}

\begin{proof}
Let $x,y\in C$ and $f\in V_C^+$. Write
\begin{align*}
x &= \sum_{i=1}^m \alpha_ix_i+v,\\
y &= \sum_{j=1}^n \beta_j y_j+w,
\end{align*}
with $v,w\in \Idem(C)$ and $x_i,y_j\in X$ as in Proposition \ref{Xrepresentation}.  Then
\[
x+y=\sum_{i=1}^m \alpha_ix_i+\sum_{j=1}^n \beta_j y_j+v+w.
\]
Observe that $\s(x_i),\s(y_j)\leq v+w$ and that $\alpha_i,\beta_j\in (0,\infty)$ for all $i,j$.
Thus, the sum on the right side is the representation of $x+y$ described in
Proposition \ref{Xrepresentation}, except for the possible repetition of elements of $X$ appearing both among the  $x_i$s and the $y_j$s. If $v+w\leq \supp(f)$, then $v\leq \supp(f)$ and $w\leq \supp(f)$, and
so 
\[
(x,f)+(y,f)=\sum_{i=1}^m \alpha_i f(x_i) + \sum_{j=1}^n \beta_jf(y_j)=(x+y,f). 
\]
If, on the other hand, $v+w\nleq \supp(f)$, then either $v\nleq \supp(f)$ or $w\nleq \supp(f)$,
and in either case $(x,f)+(y,f)=\infty=(x+y,f)$. This proves additivity on the first coordinate. 
Homogeneity with respect to scalar multiplication follows automatically from additivity.

Let $f,g\in V^+_C$ and $w\in \Idem(C)$. Then $w\le\supp(f+g)$ if and only if $w\le\supp(f)$ and 
$w\le\supp(g)$ (Lemma \ref{supportV}). This readily shows linearity on the second coordinate.

For each $x\in C$, let $\Lambda_x\in \mathrm{Hom}(V_C^+,[0,\infty])$ be defined by the pairing above: $\Lambda_x(f)=(x,f)$ for all $f\in V_C^+$. Let $\Lambda\colon C\to \mathrm{Hom}(V_C^+,[0,\infty])$ be the map given by $y\mapsto\Lambda_y$ for all $y\in C$. To prove that $\Lambda$ is injective, suppose that $y,z\in C$ are such that $\Lambda_y=\Lambda_z$. Choose any $f\in V^+_C$ such that $\supp(f)=\s(y)$. If $\s(y)\not\le \s(z)$ then $\Lambda_y(f)$ is finite, while $\Lambda_z(f)=\infty$. This contradicts that $\Lambda_y=\Lambda_z$. Hence $\s(y)\le \s(z)$. By a similar argument $\s(z)\le \s(y)$, and so we get equality. Set $w=\s(y)=\s(z)$. Then we can write 
\begin{align*}
y &=\sum_{i=1}^m\alpha_iy_i+w,\\
z &=\sum_{i=1}^n\beta_iz_i+w
\end{align*}
with $y_i,z_i\in P_{w}$ for all $i$. Let $f\in V_C$ be such that $\supp(f)=w$. Then 
$f(y_i),f(z_i)>0$ and
\begin{equation}\label{Lambdayz}
\sum_{i=1}^{m}\alpha_if(y_i)=\Lambda_y(f)=\Lambda_z(f)=\sum_{i=1}^{n}\beta_if(z_i).
\end{equation} 
Let $V_{w}^+= \{f\in\ V^+_C\colon \supp(f)=w\}$, i.e., $f\in V_w^+$
if $f$ is positive on $P_w$ and zero outside $O_w$. It is clear that  $V_w^+-V_w^+$ consists of all the functions
on $X$ that vanish outside $O_w$. It then follows from \eqref{Lambdayz}  that  $n=m$ and that, up to relabelling, $y_i=z_i$ for all $1\le i\le n$. Consequently $y=z$.

Let us show that $\Lambda$ is surjective. Let $\lambda\in \mathrm{Hom}(V_C^+,[0,\infty])$. By Lemma \ref{supportV}, the set 
\[
\{w\in \Idem(C): w=\supp(f)\hbox{ for some }f\in V_C^+\hbox{ such that }\lambda(f)<\infty\}
\] 
is closed under infima. Since this set is also finite, it has a minimum element $w$. We claim that for each $f\in V_C^+$ we have
\[
\lambda(f)<\infty\Leftrightarrow w\leq \supp(f).
\]
Indeed, from the definition of $w$ it is clear that if $\lambda(f)<\infty$ then $w\leq \supp(f)$. Suppose on the other hand that $f\in V_C^+$ is such that $w\leq \supp(f)$. Let $f_0\in V_w^+$ be such that $\lambda(f_0)<\infty$. Then $\alpha f_0 - f$ is positive (with support $w$) for a sufficiently large scalar $\alpha\in (0,\infty)$. Thus, $\lambda(f)\leq \alpha\lambda(f_0)<\infty$.

Let us extend $\lambda$ by linearity to the vector subspace $V_w:=V_w^+-V_w^+$. As remarked above, $V_{w}$ consists of all the functions $f\colon X\to \R$ vanishing on the complement of  $O_w$. That is, $V_{w}=\mathrm{span}(\{\mathbbm{1}_x\colon x\in O_w\})$, where $\mathbbm{1}_x$ denotes the characteristic function of $\{x\}$. 

If $x\in P_w$, then $\mathbbm{1}_x+\epsilon\mathbbm{1}_{P_w}\in V_w^+$ for all $\epsilon>0$; here $\mathbbm{1}_{P_w}$ denotes  the characteristic function of $P_w$. It follows that $\lambda(\mathbbm{1}_x+\epsilon\mathbbm{1}_{P_w})\geq 0$, and letting $\epsilon\to 0$, that $\lambda(\mathbbm{1}_x)\geq 0$ for all $x\in P_w$. If  $x\in O_w\backslash P_w$, then $\lambda(\mathbbm{1}_{P_w}-\alpha\mathbbm{1}_x)\geq 0$ for  all $\alpha\in \R$. It follows that $\lambda(\mathbbm{1}_x)=0$ for all $x\in O_w\setminus P_w$. Thus 
\[
\lambda(f)=\sum_{x\in P_{w}}\lambda(\mathbbm{1}_x)f(x)
\]
for all $f\in V_w$. Since $V_v^+\subseteq V_w$ for any idempotent $v$ such that $w\leq v$, the formula above holds for all $f\in V_C^+$ such that $w\leq \supp(f)$. 

Define
\[
y=\sum_{x\in P_{w}}\lambda(\mathbbm{1}_x)x+w. 
\]
By the previous arguments,  $\lambda(f)=\Lambda_y(f)$ for all $f$ such that $w\leq \supp(f)$. On the other hand, $\lambda(f)=\infty=\Lambda_y(f)$ for all $f$ such that $w\nleq \supp(f)$. Hence, $\lambda=\Lambda_y$.
\end{proof}

\bibliographystyle{plain}
\bibliography{ExtCho}

\end{document}